\documentclass[12pt, a4paper]{amsart}
\usepackage{amsfonts, amssymb, amsmath, color,hyperref}
\usepackage{amsthm}
\usepackage[english]{babel}
\usepackage[latin1]{inputenc}

\theoremstyle{plain}

	\newtheorem{lemma}{Lemma}[section]
	\newtheorem*{lemma*}{Lemma}
	\newtheorem{theorem}[lemma]{Theorem}
	\newtheorem*{theorem*}{Theorem}
	\newtheorem{corollary}[lemma]{Corollary}
	\newtheorem*{corollary*}{Corollary}
	\newtheorem{proposition}[lemma]{Proposition}
	\newtheorem*{proposition*}{Proposition}
	\newtheorem{claim}{Claim}
	\newtheorem*{claim*}{Claim}
	\newtheorem{remark}[lemma]{Remark}
	\newtheorem*{remark*}{Remark}
	\newtheorem*{conjecture*}{Conjecture}
	
	\newtheorem{fact}[lemma]{Fact}
	\newtheorem{definition}[lemma]{Definition}
	\newtheorem{notation}[lemma]{Notation}

\let\oldtocsection=\tocsection

\let\oldtocsubsection=\tocsubsection

\let\oldtocsubsubsection=\tocsubsubsection

\renewcommand{\tocsection}[2]{\hspace{0em}\oldtocsection{\it #1}{#2}}
\renewcommand{\tocsubsection}[2]{\hspace{1em}\oldtocsubsection{\it #1}{\textit #2}}
\renewcommand{\tocsubsubsection}[2]{\hspace{2em}\oldtocsubsubsection{\it #1}{\it#2}}

\makeatletter
\def\@tocline#1#2#3#4#5#6#7{\relax
  \ifnum #1>\c@tocdepth 
  \else
    \par \addpenalty\@secpenalty\addvspace{#2}%
    \begingroup \hyphenpenalty\@M
    \@ifempty{#4}{%
      \@tempdima\csname r@tocindent\number#1\endcsname\relax
    }{%
      \@tempdima#4\relax
    }%
    \parindent\z@ \leftskip#3\relax \advance\leftskip\@tempdima\relax
    \rightskip\@pnumwidth plus4em \parfillskip-\@pnumwidth
    #5\leavevmode\hskip-\@tempdima
      \ifcase #1
       \or\or \hskip 1em \or \hskip 2em \else \hskip 3em \fi%
      #6\nobreak\relax
    \dotfill\hbox to\@pnumwidth{\@tocpagenum{#7}}\par
    \nobreak
    \endgroup
  \fi}
\makeatother
\begin{document}

\bigskip

\title{Two-dimensional simply connected abelian locally Nash groups}

\author{E. Baro}
\address{Departamento de \'Algebra, Facultad de Ciencias Matem\'aticas, Universidad Complutense de Ma\-drid, 28040 Madrid (Spain)}
\curraddr{}
\email{eliasbaro@pdi.ucm.es}
\author{J. de Vicente}
\author{M. Otero}
\address{Departamento de Matem\'aticas, Universidad Aut\'onoma de Madrid, 28040 Madrid (Spain)}
\curraddr{}
\email{juan.devicente@uam.es, margarita.otero@uam.es}

\thanks{First and third authors supported by Spanish GAAR MTM2011-22435 and MTM2014-55565-P. Second author supported by a grant of the International Program of Excellence in Mathematics at Universidad Aut\'onoma de Madrid.
The results of this paper are part of the second author's Ph.D. dissertation.}


\date{\today}

\begin{abstract}
The aim of this paper is to give a description of simply connected abelian locally Nash groups of dimension $2$. 
Along the way we prove that, for any $n\geq 2$, a locally Nash structure over $(\mathbb{R}^n,+)$ can be characterized via a meromorphic map admitting an algebraic addition theorem.
\end{abstract}
\maketitle

\vspace*{-1cm}
{\footnotesize
\tableofcontents
}
\vspace*{-1cm}

\section{Introduction.}\label{Introduction}

In $1952$ John Nash proved in \cite{Nash} that any compact smooth manifold may be equipped with both an analytic and a 
semialgebraic structure.
After Nash's article, those analytic manifolds that are also equipped with a semialgebraic structure are named Nash manifolds.
These manifolds combine the good properties of analytic manifolds together with the finiteness properties of semialgebraic manifolds,
while remaining complex enough to present interesting problems.
Because of that, Nash manifolds have attracted the attention of many mathematicians, being \cite{Bochnak_Coste_Roy} and \cite{Shiota1} 
the main references and a good introduction to the subject.
Among Nash manifolds, of special interest are the Nash groups, which are analytic groups admitting a semialgebraic structure.
The most relevant result about Nash groups is done by Hrushovski and Pillay in \cite{Hrushovski_Pillay} 
(see also \cite{Hrushovski_Pillay_Errata}), where a close relation between affine Nash groups and real algebraic groups is established.
However, although Nash groups share some of the good properties of algebraic groups, not much else is known 
about them, specially in the non-affine case.
Shiota reviews the main results on Nash manifolds in \cite{Shiota2}, including a description of the one-dimensional 
affine Nash groups given in  \cite{Madden_Stanton} (see also \cite{Madden_Stanton_Errata}).
In order to obtain that description, it is essential to get first a description of locally Nash groups, which are analytic groups admitting a 
``weak'' semialgebraic structure.
The semialgebraic structure is weakened in order to allow the universal coverings of Nash groups -- which are not in general Nash groups -- to be locally Nash groups. Note that Nash groups are a particular case of groups definable in o-minimal structures, see, e.g., \cite{Otero} and its 
references for literature about these groups.

The purpose of this article is to give a description of the locally Nash group structures on $(\mathbb{R}^2,+)$ and also to clarify the 
category of locally Nash groups.
We see this as a first step to obtain a description of the two-dimensional abelian Nash groups.

Next, we introduce the category of locally Nash groups.
Given an open subset $U$ of $\mathbb{R}^n$, a map $f=(f_1,\ldots ,f_m):U\rightarrow \mathbb{R}^m$ is a \emph{Nash map} if 
$f_1,\ldots ,f_m$ are both analytic and semialgebraic.
A \emph{locally Nash manifold} is an analytic manifold $M$ with an atlas $\{(U_i,\phi _i)\}_{i\in I}$ such that for each 
$i,j\in I$, $\phi _i(U_i\cap U_j)$ is semialgebraic and the transition maps are Nash maps.
We call such atlas a \emph{Nash atlas}.
A map $f:M\rightarrow N$ between locally Nash manifolds is a \emph{locally Nash map} if for each point $x$ of $M$ and $f(x)$ of $N$
there exist charts $(U,\phi)$ and $(V,\psi)$ of their respective Nash atlases and an open subset $U'$ of $U$ such that $x\in U'$, 
$f(U')\subset V$ and 
$\psi \circ f\circ \phi ^{-1} : \phi (U') \rightarrow \psi (V)$
is a Nash map.
In a natural way we define locally Nash group and locally Nash group homomorphism/isomorphism between Nash groups (see Section 
\ref{category}).
The study of Nash atlases for $(\mathbb{R}^n,+)$ will lead us in a natural way to the concept of algebraic addition theorem, that we now recall.
Let $\mathbb{K}$ be $\mathbb{C}$ or $\mathbb{R}$.
Let $A_{\mathbb{K},n}$ be the ring of all power series in $n$ variables with coefficients in $\mathbb{K}$ that are convergent in a 
neighborhood of the origin.
Let $M_{\mathbb{K},n}$ be the quotient field of $A_{\mathbb{K},n}$.
Let $u$ and $v$ be variables of $\mathbb{C}^n$.
We say $(\phi _1,\ldots ,\phi _n)\in M_{\mathbb{K},n}^n$ admits an \emph{algebraic addition theorem} (\emph{AAT}) if 
$\phi _1,\ldots ,\phi _n$ are algebraically independent over $\mathbb{K}$ and for each $i$ each $\phi _i (u+v)$ is algebraic over 
$\mathbb{K}(\phi _1(u),\ldots ,\phi _n(u),\phi _1(v),\ldots ,\phi _n(v))$. Note that the AAT is independent of $\mathbb{K}$.
 The $n$ coordinate functions of a Nash coordinate neighborhood of a local Nash group structure on $(\mathbb{R}^n,+)$ admits 
naturally (a functional version of) AAT (see Lemma \ref{AAT-star}).
This will be specially useful to study the different locally Nash structures of $(\mathbb{R}^n,+)$.
We remark that although $(\mathbb{R}^n,+)$ has a unique analytic structure (the standard analytic structure, {\it i.e.} the one 
compatible with the identity map), it may have several different locally Nash structures, and the main aim of this paper is to describe them for $n=2$.

One of our main results is Theorem \ref{T1}, that states that for each locally Nash structure on $(\mathbb{R}^n,+)$ there
exists $f:\mathbb{C}^n\rightarrow \mathbb{C}^n$ such that:

\smallskip
\hspace{-0.4cm}$(1)$ $f$ is a real meromorphic map, {\it i.e.} each of its coordinate functions is the quotient of two analytic functions $h_1$ and $h_2$
satisfying $h_i(\mathbb{R}^n)\subset \mathbb{R}$ ($i=1,2$), and\\
 $(2)$ there exist $k\in \mathbb{R}^n$ and an open neighborhood $U\subset \mathbb{R}^n$ of $0$ such that 
 \[\phi: U\rightarrow \mathbb{R}^n: u\mapsto f(u+k)\]
 is an analytic diffeomorphism, its image is semialgebraic and the Taylor power series expansion of $\phi $ at $0$ admits an AAT.
\smallskip

\noindent
Moreover, the translates of $(U,\phi )$ give a locally Nash group structure isomorphic to the original one.
We denote this locally Nash structure by $(\mathbb{R}^n,+,f)$ (note that this notation is consistent with that of 
\cite{Madden_Stanton}).
More precisely, we shall prove the following:

\noindent {\bf Theorem \ref{T1} }
{\em Every simply connected $n$-dimensional abelian locally Nash group is locally Nash isomorphic to some
$(\mathbb{R}^n,+,f)$ where $f:\mathbb{C}^n\rightarrow \mathbb{C}^n$ is a real meromorphic map admitting an AAT.}

To prove Theorem \ref{T1} we will make use of the following result, which might be of interest by itself.
In fact, this result follows the lines of the classical work of Weierstrass for functions admitting an AAT (see \cite[Ch. XXI]{Hancock} for
dimension $1$ and \cite{Painleve} for a general discussion of the problem on higher dimensions).

\noindent {\bf Theorem \ref{T2}}
{\em Let $\phi _1,\ldots ,\phi _n \in M_{\mathbb{K},n}$ admitting an AAT.
Then there exists $\psi _1,\ldots ,\psi _n\in M_{\mathbb{K},n}$ admitting an AAT and algebraic over 
$\mathbb{K}(\phi _1,\ldots ,\phi _n)$ and $\psi _0\in M_{\mathbb{K},n}$ algebraic over 
$\mathbb{K}(\psi _1,\ldots ,\psi _n)$ such that:\\
$(1)$ for each $f\in \mathbb{K}(\psi _0,\ldots ,\psi _n)$ there exists $R\in \mathbb{K}(X_1,\ldots ,X_{2(n+1)})$ such that
\[
f (u+v)=R \big(\psi _0(u),\ldots ,\psi _n(u),\psi _0(v),\ldots ,\psi _n(v)\big)
\]
and\\
$(2)$ each $\psi _0,\ldots ,\psi _n$ is the quotient of two power series (of $A_{\mathbb{K},n}$),
both convergent in all $\mathbb{C}^n$.
}

In order to describe the locally Nash groups structures on $(\mathbb{R}^2,+)$, Theorem \ref{T1} allow us to use the description 
given by Painlev\'e in \cite{Painleve} of pairs of meromorphic functions on $\mathbb{C}^2$ which admit an AAT.
The description is based on the Weierstrass functions $\wp _\Omega$, $\zeta _\Omega$ and $\sigma _\Omega$ 
corresponding to a lattice $\Omega$ of $(\mathbb{C},+)$ and on the fields of abelian functions $\mathbb{C}(\Lambda )$ corresponding to a 
lattice $\Lambda$ of $(\mathbb{C}^2,+)$
({\it i.e.} $f\in \mathbb{C}(\Lambda )$ if and only if  $f:\mathbb{C}^2\rightarrow \mathbb{C}$ is a meromorphic function such that 
$f(z+\lambda )=f(z)$ for all $\lambda \in \Lambda$).
 Painlev\'e proves in \cite{Painleve} that a pair of meromorphic functions from $\mathbb{C}^2$ to $\mathbb{C}$ which admits 
an AAT is a transcendence basis of a field belonging to one of the families $\mathcal{P}_1$,\ldots ,$\mathcal{P}_6$ (which we will call  \emph{the families of the Painlev\'e description}) given in the statement below.

\noindent {\bf Theorem \ref{T3}}
{\em Every simply connected $n$-dimensional abelian locally Nash group is locally Nash isomorphic to one of the form 
$(\mathbb{R}^2,+,f)$ where $f:\mathbb{C}^2\rightarrow \mathbb{C}^2$ is a real meromorphic map admitting an AAT and such 
that its coordinate functions are algebraic over one of the fields of the following families:
\begin{enumerate}
\item[$(1)$] $\mathcal{P}_1 :=\{ \, \mathbb{C}(g_1\circ \alpha ) : \alpha \in GL_2(\mathbb{C})\, \}$, where $g_1(u,v)=(u,v)$; \vspace{1mm}
\item[$(2)$] $\mathcal{P}_2 :=\{ \, \mathbb{C}(g_2\circ \alpha ) : \alpha \in GL_2(\mathbb{C})\, \}$, where $g_2(u,v)=(u,e^v)$; \vspace{1mm}
\item[$(3)$] $\mathcal{P}_3 :=\{ \, \mathbb{C}(g_3\circ \alpha ) : \alpha \in GL_2(\mathbb{C})\, \}$, where $g_3(u,v)=(e^u,e^v)$; \vspace{1mm}
\item[$(4)$] $\mathcal{P}_4 :=\{ \, \mathbb{C}(g_{4 ,a, \Omega}\circ \alpha ) : \alpha \in GL_2(\mathbb{C}),\, a \in \{0,1\}, 
\,\Omega \text{ is a lattice of } (\mathbb{C},+)\, \}$, where $g_{4,a ,\Omega}(u,v)=(\wp _{\Omega}(u),v-a\zeta _{\Omega} (u))$; \vspace{1mm}
\item[$(5)$] $\mathcal{P}_5 :=\{ \, \mathbb{C}(g_{5,a ,\Omega}\circ \alpha ) : \alpha \in GL_2(\mathbb{C}),\, a \in \mathbb{C}, 
\,\Omega \text{ is a lattice of } (\mathbb{C},+)\, \}$, where $g_{5,a ,\Omega}(u,v)=
\left(\wp _{\Omega}(u),\frac{\sigma _{\Omega} (u-a)}{\sigma _{\Omega} (u)}e^v\right)$; and \vspace{1mm}
\item[$(6)$] $\mathcal{P}_6:=\{  \, \mathbb{C}(\Lambda ) : \Lambda \text{ is a lattice of } (\mathbb{C}^2,+),\,
tr.deg.\, _{\mathbb{C}}\mathbb{C}(\Lambda )=2 \}$.
\end{enumerate}
Furthermore, if $(\mathbb{R}^2,+,g)$ is another locally Nash group, where $g:\mathbb{C}^2\rightarrow \mathbb{C}^2$ is a real meromorphic 
map admitting an AAT, and the coordinate functions of $f$ and $g$ are algebraic over fields of different families, then $(\mathbb{R}^2,+,f)$ and 
$(\mathbb{R}^2,+,g)$ are not locally Nash isomorphic.

Even further, each of the families induce at least one locally Nash group structure on $(\mathbb{R}^2,+)$.
}

The sections of the article are divided as follows:
in Section \ref{category} we define the category of locally Nash groups, in Section \ref{AAT} we prove the basic properties of AAT 
and Theorem \ref{T2}, in Section \ref{meromorphic maps} we extend the results of \ref{AAT} for meromorphic functions and we prove 
Theorem \ref{T1} and finally in Section \ref{section2d} we prove Theorem \ref{T3}.
We also include an appendix where we rewrite the proof of the classification of the one-dimensional simply connected locally Nash groups (\cite[Theorem 1]{Madden_Stanton}) in a uniform way.

\section{Category of Locally Nash Groups.}\label{category}

In this section we collect the definitions and basic properties related to locally Nash manifolds and groups.

\medskip

\subsection{Locally Nash manifolds.} Let $U$ be an open subset of $\mathbb{R}^m$.
We say that $f:U\rightarrow \mathbb{R}^n$ is a \emph{Nash map} if $f$ is both semialgebraic and analytic.
Alternatively, a Nash map can be described as follows.
Given maps $f:W\rightarrow \mathbb{R}^n$ and $g:W\rightarrow \mathbb{R}^n$ we say that $g$ is \emph{algebraic} over $\mathbb{R}(f)$ 
on $W$ if for each $i\in \{1,\ldots ,n\}$ there exists a polynomial 
$P_i\in \mathbb{R}[X_1,\ldots ,X_n,Y]$ of positive degree in $Y$ such that 
\[
P_i(f_1(x),\ldots ,f_n(x), g_i(x))\equiv 0 \text{ on } W. 
\]
Let $U$ be an open subset of $\mathbb{R}^m$.
Then, $f:U\rightarrow \mathbb{R}^n$ is a Nash map if and only if $U$ is semialgebraic, $f$ is analytic and $f(x)$ is 
algebraic over $\mathbb{R}(x)$ on $U$ (see \cite[Proposition 8.1.8]{Bochnak_Coste_Roy}).
In all what follows we will make use of this characterization without further mention.
We say that $f:U\rightarrow V\subset \mathbb{R}^n$ is a \emph{Nash diffeomorphism} if $f$ is an analytic 
diffeomorphism and both $f$ and $f^{-1}$ are Nash maps.
Let $M$ be an analytic manifold. 
Two charts $(U,\phi )$ and $(V,\psi )$ of an atlas for $M$ are \emph{Nash compatible} if $\phi (U)$ and $\psi (V)$ are
semialgebraic and either $U\cap V= \emptyset$ or 
\[
\psi \circ \phi ^{-1}:\phi (U\cap V)\rightarrow \psi (U\cap V) 
\]
is a Nash diffeomorphism.
An atlas of $M$ is a \emph{Nash atlas} if any two charts in the atlas are Nash compatible.
In particular, $\phi (U)$ is semialgebraic for any $(U,\phi )$ in the Nash atlas.
An analytic manifold $M$ together with a Nash atlas is called a \emph{locally Nash manifold}.

\begin{definition}\label{locally Nash map}
\emph{
Let $M_1$ and $M_2$ be locally Nash manifolds equipped with Nash atlases $\{(U_i,\phi _i)\}_{i\in I}$ and 
$\{(V_j,\psi _j)\}_{j\in J}$ respectively.
A \emph{locally Nash map} $f:M_1\rightarrow M_2$ is a (continuous) map such that for every $p\in M_1$ and every 
$j\in J$ such that $f(p)\in V_j$ there exists $i\in I$ and an open subset $U\subset U_i$ such that 
$p\in U$, $f(U)\subset V_j$ and
\[
\psi _j \circ f\circ \phi _i ^{-1} : \phi _i  (U) \rightarrow \psi _j (V_j)
\]
is a Nash map.
(For an equivalent definition see Proposition \ref{characterization of locally Nash maps}.)}
\end{definition}

\noindent A locally Nash map $f:M_1\rightarrow M_2$ is a \emph{locally Nash diffeomorphism} if $f$ is an analytic (global) 
diffeomorphism and both $f$ and $f^{-1}$ are locally Nash maps.
A \emph{locally Nash group} is a locally Nash manifold equipped with group operations (multiplication and inversion) which 
are given by locally Nash maps.
A \emph{homomorphism of locally Nash groups} is a locally Nash map that is also a homomorphism of groups. 
An \emph{isomorphism of locally Nash groups} is a locally Nash diffeomorphism that is also an isomorphism of groups.
Clearly a map $f$ is an isomorphism of locally Nash maps if and only if both $f$ is an isomorphism 
of abstract groups and $f$ and $f^{-1}$ are locally Nash maps.

\medskip

Locally Nash maps can be characterized as follows.

\begin{proposition}\label{characterization of locally Nash maps} Let $M_1$ and $M_2$ be locally Nash manifolds with Nash atlases 
$\{(U_i,\phi _i)\}_{i\in I}$ and $\{(V_j,\psi _j)\}_{j\in J}$ respectively.
The following are equivalent:
\begin{enumerate}
 \item[$(1)$] $f:M_1\rightarrow M_2$ is a locally Nash map.
 \item[$(2)$] For every $p\in M_1$ and for each $i\in I$ and $j\in J$ such that $p\in U_i$ and $f(p)\in V_j$ 
 there exists an open subset $U$ of $U_i$ such that $p\in U$, $f(U)\subset V_j$, and 
 \[ 
 \psi _j\circ f \circ \phi _i^{-1} :\phi _i(U)\rightarrow \psi _j(V_j) 
 \]
 is a Nash map.
 \item[$(3)$] For every $p\in M_1$ there exist $i \in I$ and $j\in J$ such that $p\in U_i$ and $f(p)\in V_j$ and 
 there exists an open subset $U$ of $U_i$ such that $p\in U$, $f(U)\subset V_j$, and 
 \[
 \psi_j \circ f \circ \phi_i^{-1}:\phi_i(U )\rightarrow \psi_j (V_j ) 
 \] 
 is a Nash map.
\end{enumerate}
\end{proposition}
\begin{proof}
Since $(2)$ implies $(1)$ and $(1)$ implies $(3)$, it is enough to show that $(3)$ implies $(2)$.
Fix $p\in M_1$ and let $i \in I$, $j\in J$ and $U\subset U_i$ whose existence ensures $(3)$.
Fix $k\in I$ and $\ell \in J$ with $p\in U_k$ and $f(p)\in V_\ell$.
Clearly, it suffices to show that there exists an open subset $U'$ of $U_k$ with $p\in U'$ such that 
\[
\psi_\ell \circ f \circ \phi_k^{-1}:\phi_k(U')\rightarrow \psi_\ell (V_\ell) 
\]
is Nash.
To prove the latter, firstly note that $\psi_j \circ f \circ \phi_i^{-1}$ is continuous and both $U_i\cap U_k\ni p$ and 
$V_j\cap V_\ell\ni f(p)$ are open, hence there exists an open subset $U'$ of $U_i\cap U_k$ with $p\in U'$ such that
\[
(\psi_j \circ f \circ \phi_i^{-1})(\phi_i(U'))\subset \psi_j(V_j\cap V_\ell). 
\]
Moreover, we can assume that $\phi_i(U')$ is semialgebraic (it suffices to take, instead of $U'$, the preimage of an open ball 
centered in $\phi_i(p)$ and contained in the original $\phi_i(U')$).
In particular, since the restriction of a Nash map to an open semialgebraic set is a Nash map, the map 
\[
\psi_j \circ f \circ \phi_i^{-1}:\phi_i(U')\rightarrow \psi_j (V_j\cap V_\ell) 
\]
is still a Nash map.
On the other hand, both change of charts 
\[
\phi_i \circ \phi_k^{-1}:\phi_k(U') \rightarrow \phi_i(U')
\]
and
\[
\psi_\ell \circ \psi_j^{-1}:\psi_j(V_j\cap V_\ell)\rightarrow \psi_\ell(V_j\cap V_\ell) 
\]
are Nash maps.
Thus, the composition of the last three maps,
\[
\psi_\ell \circ f \circ \phi_k^{-1}=(\psi_\ell \circ \psi_j^{-1})\circ (\psi_j \circ f \circ \phi_i^{-1})\circ (\phi_i \circ \phi_k^{-1})
: \phi _k(U')\rightarrow \psi _\ell(V_\ell)
\]
is a Nash map, as required.
\end{proof}

From Proposition \ref{characterization of locally Nash maps}.$(2)$ it is clear that the composition of locally Nash maps is a 
locally Nash map.
We also deduce the following.

\begin{lemma}\label{locally Nash diffeomorphism} Let $M_1$ and $M_2$ be locally Nash manifolds.
Then $f:M_1\rightarrow M_2$ is a locally Nash diffeomorphism if and only if $f$ is both an analytic diffeomorphism and a locally 
Nash map.
\end{lemma}
\begin{proof}
We show the nontrivial implication.
Let $\{(U_i,\phi _i)\}_{i\in I}$ and $\{(V_j,\psi _j)\}_{j\in J}$ be the Nash atlases of $M_1$ and $M_2$ respectively.
We have to show that $f^{-1}:M_2\rightarrow M_1$ is a locally Nash map.
Fix $p\in M_2$ and $i\in I$ such that $f^{-1}(p)\in U_i$.
We have to show that there exists $j\in J$ and an open subset $V\subset V_j$ such that $p\in V$, $f^{-1}(V)\subset U_i$, 
$\psi _j(V)$ is semialgebraic and 
\[
\phi _i\circ f^{-1}\circ \psi _j^{-1}:\psi _j(V)\rightarrow \phi _i(U_i) 
\]
is a Nash map.
Let $j\in J$ be such that $p\in V_j$.
For these $f^{-1}(p)\in M_2$, $i$ and $j$, since $f$ is a locally Nash map, we can apply Proposition 
\ref{characterization of locally Nash maps}.$(2)$ and get an open subset $U$ of $U_i$ such that $f^{-1}(p)\in U$,
$f(U)\subset V_j$ and 
\[
\psi _j\circ f\circ \phi _i ^{-1}: \phi _i(U)\rightarrow \psi _j(V_j) 
\]
is a Nash map.
Therefore, the given $j$ and $V:=f(U)$ satisfy the required conditions once we note that the inverse of a bijective 
semialgebraic map is a semialgebraic map.
\end{proof}

\subsection{Locally Nash groups} Next, we show that to describe the locally Nash structure of a locally Nash group it is enough to do it near the identity.
We introduce new notations that will be useful for this purpose.
Let $(G,\cdot)$ equipped with an analytic atlas $\mathcal{A}$ be an analytic group -- thus a Lie group -- and let $(U,\phi)$ be a 
chart of the identity of $\mathcal{A}$.
From the theory of analytic groups we recall that
\[
\mathcal{A}_{(U,\phi)}:= \{ (gU, \phi _g)\ | \ \phi _g:gU\rightarrow \mathbb{R}^n:u\mapsto \phi (g^{-1}u)\}_{g\in G}
\]
is also an analytic atlas for $(G,\cdot)$.
We will keep the notation $\mathcal{A}_{(U,\phi)}$ for this canonical atlas.
In the above example, $\mathcal{A}_{(U,\phi)}$ might not be a Nash atlas for $(G,\cdot)$, but if it is so then, the locally 
Nash group $(G,\cdot)$ equipped with $\mathcal{A}_{(U,\phi)}$ will be denoted $(G,\cdot,\phi |_U)$, 
see Fact \ref{compatibility0} and Proposition \ref{compatibility1}.
(The notation $(\mathbb{R}^n,+,f)$, where $f:\mathbb{C}^n\rightarrow \mathbb{C}^n$ is as mentioned in the introduction, 
will be justified in Section \ref{meromorphic maps} once Lemma \ref{AAT-star} is proved.)

\begin{fact}[{\cite[Lemma 1]{Madden_Stanton}}]\label{compatibility0}
Let $(G,\cdot)$ be an analytic group with atlas $\mathcal{A}$.
Let $(U,\phi )\in \mathcal{A}$ be a chart of the identity such that:
\begin{enumerate}
\item[$(i)$] there exists an open neighborhood of the identity $U'\subset U$ such that
\[
\phi \circ \cdot \circ (\phi ^{-1},\phi ^{-1}):\phi (U')\times \phi (U')\rightarrow \phi (U) :
(x,y)\mapsto \phi (\phi ^{-1}(x)\cdot \phi ^{-1}(y))
\]
is a Nash map, and
\item[$(ii)$] for each $g\in G$ there exists an open neighborhood of the identity $U_g\subset U$ such that
\[
\phi \circ \ ^{-1} \circ \phi ^{-1}:\phi (U_g)\rightarrow \phi (U) :
x\mapsto \phi (g^{-1}\phi ^{-1}(x)g)
\]
is a Nash map.
\end{enumerate}
Then there exists $V\subset U$ such that $\mathcal{A}_{(V,\phi )}=\{ (gV,\phi _g)\} _{g\in G}$ is a Nash atlas for $(G,\cdot)$ and 
hence $(G,\cdot ,\phi |_V)$is a locally Nash group.
\end{fact}

We note that when $(G,\cdot)$ is an abelian group then $(ii)$ of Fact \ref{compatibility0} is trivially 
satisfied.
So, in this case, the proposition says that each chart of the identity satisfying $(i)$ induces a locally Nash group 
structure on $(G,\cdot)$.
We anticipate from Lemma \ref{AAT-star} that a chart of the identity $(U,(\phi _1,\ldots ,\phi _n))$ 
of $(\mathbb{R}^n,+)$ with its standard analytic structure satisfies $(i)$ if and only if it admits an algebraic 
addition theorem, {\it i.e.} if for some open neighborhood of the identity $U'\subset U$ and for each $i\in \{1,\ldots ,n\}$ 
there exists a $P_i\in \mathbb{K}[X_1,\ldots ,X_{2n+1}]$, $P_i\neq 0$, such that
\[
P_i(\phi _1(u),\ldots ,\phi _n(u),\phi _1(v),\ldots ,\phi _n(v),\phi _i(u+v))\equiv 0 \text{ on } U'\times U'.
\]

\begin{proof}[Proof of Fact \ref{compatibility0}]
Firstly, given $(U,\phi )\in \mathcal{A}$, a chart of the identity satisfying $(i)$ and $(ii)$, we will find $V\subset U$ such 
that $G$ equipped with $\mathcal{A}_{(V,\phi)}:=\{ (gV, \phi _g)\}_{g\in G}$ where 
\[
\phi _g:gV\rightarrow \mathbb{R}^n: u\mapsto \phi _g(u)=\phi (g^{-1}u) 
\]
is a locally Nash manifold (for this only $(i)$ is needed).
Then, we will check that $\cdot :G\times G\rightarrow G$ is a locally Nash map when $G$ is equipped with $\mathcal{A}_{(V,\phi )}$.
Finally, we will show that $^{-1}:G\rightarrow G$ is a locally Nash map when $G$ is equipped with $\mathcal{A}_{(V,\phi )}$.
This will complete the proof.

\smallskip

Since the map of $(i)$ is continuous, there exists an open 
neighborhood of the identity $V\subset U'$ such that $V\cdot V\subset U'$ and $V=V^{-1}$.
Moreover, we can assume that $\phi (V)$ is semialgebraic (it suffices to take the preimage of an open ball centered in $\phi$ of
the identity and contained in the original $\phi (U)$).
We show that $\mathcal{A}_{(V,\phi )}$, as defined above, is a Nash atlas for $G$.
We note that for each $g\in G$
\[
(\phi _g)^{-1}: \phi (V)\rightarrow gV : x\mapsto g\phi ^{-1}(x). 
\]
So we have to check that if $g,h\in G$ are given such that $gV\cap hV\neq \emptyset$ then
\[
\phi _h\circ (\phi _g)^{-1}:\phi (V\cap g^{-1}hV)\rightarrow \phi (V\cap h^{-1}gV): x\mapsto \phi (h^{-1}g\phi ^{-1}(x))
\]
is a Nash diffeomorphism.
Since $V\cdot V\subset U'$ and $V=V^{-1}$, we have that $h^{-1}g\in U'$.
Semialgebraic sets are closed under projections, thus we can evaluate the map of $(i)$ at $(\phi (h^{-1}g),x)$ to deduce that 
\[
\phi _h\circ (\phi _g)^{-1}:\phi (U')\rightarrow \phi (U):x\mapsto \phi (h^{-1}g\phi ^{-1}(x))
\]
is a Nash map.
Since $\phi (h^{-1}gV)$ is the image of $\phi (V)$ by $\phi _h\circ (\phi _g)^{-1}$ and $\phi (V)$ is semialgebraic, 
$\phi (h^{-1}gV)$ is also semialgebraic.
We note that $\phi (V\cap h^{-1}gV)$ is equal to $\phi (V)\cap \phi (h^{-1}gV)$ and hence semialgebraic.
So the map
\[
\phi _h\circ (\phi _g)^{-1}:\phi (V\cap g^{-1}hV)\rightarrow \phi (V\cap h^{-1}gV): x\mapsto \phi (h^{-1}g\phi ^{-1}(x))
\]
is a Nash map.
By symmetry, the same argument shows that $\phi _g\circ (\phi _h)^{-1}$ is also a Nash map.
We recall from the theory of analytic groups that $\mathcal{A}_{(V,\phi )}$ is an atlas for $G$.
This implies that $\phi _h\circ (\phi _g)^{-1}$ is an analytic diffeomorphism and hence a Nash diffeomorphism.
Therefore $G$ equipped with $\mathcal{A}_{(V,\phi)}$ is a locally Nash manifold.

\smallskip

Now we check that $\cdot :G\times G\rightarrow G$ is a locally Nash map when $G$ is equipped with $\mathcal{A}_{(V,\phi )}$.
By Proposition \ref{characterization of locally Nash maps}.$(3)$ it is enough to check that for each $g,h\in G$ there exist open 
neighborhoods of the identity $V_1,V_2\subset V$ such that
\[
\phi_{gh} \circ \cdot \circ ((\phi _g) ^{-1}\!\!	,(\phi _h) ^{-1}):\phi (V_1)\times \phi (V_2)\rightarrow \phi (V) :
(x,y)\mapsto \phi (h^{-1}\phi ^{-1}(x) h\phi ^{-1}(y))
\]
is a Nash map.
Reasoning as in the first part of the proof and since the maps of $(i)$ and $(ii)$ for $h$ are Nash,
there exist open neighborhoods of the identity $V_1',V_2\subset V$ and $V_1\subset V_1'\cap U_h$ such that both
\[
\phi \circ \cdot \circ (\phi ^{-1},\phi ^{-1}):\phi (V_1')\times \phi (V_2)\rightarrow \phi (V) :
(x,y)\mapsto \phi (\phi ^{-1}(x)\cdot \phi ^{-1}(y))
\]
and
\[
\phi \circ \ ^{-1} \circ \phi ^{-1}:\phi (V_1)\rightarrow \phi (V_1') :
x\mapsto \phi (h^{-1}\phi ^{-1}(x)h)
\]
are Nash maps.
An adequate composition - which is also Nash - of the latter maps gives the map which was required to be Nash.

\smallskip

Next we show that the map
\[
\phi \circ ^{-1} \circ \,
\phi ^{-1}:\phi (V)\rightarrow \phi (V) :
x\mapsto \phi ((\phi ^{-1}(x))^{-1})\tag{$\ast$}\label{star3}
\]
is Nash.
Since the map of (\ref{star3}) is analytic (because $\mathcal{A}$ is an analytic atlas for $(G,\cdot)$), it is enough to check that
it is semialgebraic.
Without loss of generality we may assume that $\phi$ of the identity is $0$.
We note that since the map of $(i)$ is Nash
\[
A:=\{(x,y)\in \phi (V)\times \phi (V): \phi (\phi ^{-1}(x)\cdot \phi ^{-1}(y)) =0\} 
\]
is semialgebraic.
Since each $g\in G$ has a unique inverse element and $V=V^{-1}$, it follows that
\[
A=\{(x,y)\in \phi (V)\times \phi (V): y=\phi ((\phi ^{-1}(x))^{-1}) \}.
\]
Now since the projection of a semialgebraic set is a semialgebraic set, it follows that the graph of the map (\ref{star3}) is 
semialgebraic and hence the map of (\ref{star3}) is semialgebraic, as it was required.

\smallskip

Now we check that $^{-1}:G\rightarrow G$ is a locally Nash map when $G$ is equipped with $\mathcal{A}_{(V,\phi )}$.
By Proposition \ref{characterization of locally Nash maps}.$(3)$ it is enough to prove that for each $g\in G$ there exists an open 
neighborhood of the identity $V_1\subset V$ such that
\[
\phi_{g^{-1}} \circ \ ^{-1} \circ (\phi _g)^{-1}:\phi (V_1)\rightarrow \phi (V) :
x\mapsto \phi (g(\phi ^{-1}(x))^{-1}g^{-1})
\]
is a Nash map.
Reasoning again as in the first part of the proof and since the map of property $(ii)$ for $g^{-1}$ is Nash,
there exists an open neighborhood of the identity $V_1\subset V$ such that
\[
\phi \circ \ ^{-1} \circ \phi ^{-1}:\phi (V_1)\rightarrow \phi (V) :
x\mapsto \phi (g\phi ^{-1}(x)g^{-1})
\]
is a Nash map.
Composing the latter map with the map in (\ref{star3}) we obtain a Nash map, which is the map required to be Nash.
This completes the proof.
\end{proof}

\begin{proposition}\label{compatibility1}
Let $(G,\cdot)$ be a locally Nash group equipped with a Nash atlas $\mathcal{A}$.
Then, for every chart of the identity $(U,\phi )\in \mathcal{A}$, there exists an open subset $V$ of $U$ such that $(G,\cdot)$ 
equipped with $\mathcal{A}$ is isomorphic as a locally Nash group to $(G,\cdot , \phi |_V)$.
\end{proposition}
\begin{proof}
Firstly, we will check that $(U,\phi)$ satisfies $(i)$ and $(ii)$ of Fact \ref{compatibility0}.
Then, by Fact \ref{compatibility0}, there exists $V\subset U$ such that $\mathcal{A}_{(V,\phi )}$ is a Nash atlas for 
$(G,\cdot )$.
Finally, we will show that the identity map from $G$ equipped with $\mathcal{A}$ to $G$ equipped with 
$\mathcal{A}_{(V,\phi )}$ is a locally Nash diffeomorphism, and hence an isomorphism of locally Nash groups.

\smallskip

Let $(U,\phi )\in \mathcal{A}$ be a chart of the identity.
Since $\cdot :G\times G\rightarrow G$ is a locally Nash map when $G$ is equipped with $\mathcal{A}$, by Proposition
\ref{characterization of locally Nash maps}.$(2)$ we deduce the following facts.
\begin{enumerate}
\item[$(1)$] There exists an open neighborhood of the identity $U'\subset U$ such that
\[
\phi \circ \cdot \circ (\phi ^{-1},\phi ^{-1}):\phi (U')\times \phi (U')\rightarrow \phi (U) :
(x,y)\mapsto \phi (\phi ^{-1}(x)\cdot \phi ^{-1}(y))
\]
is a Nash map.
So $(U,\phi )$ satisfies $(i)$ of Fact \ref{compatibility0}.
\item[$(2)$] Fix $g\in G$ and $(W_1,\psi _1),(W_2,\psi _2)\in \mathcal{A}$ coordinate neighborhoods of $g$ and $g^{-1}$ respectively.
Then there exist open neighborhoods $W_1'\subset W_1$ and $W_2'\subset W_2$ of $g$ and $g^{-1}$ respectively such that
\[
\begin{array}{rcrcl}
\phi \circ \cdot \circ (\psi _2 ^{-1},\psi _1 ^{-1})& : & \psi _2 (W_2')\times \psi _1 (W_1')& \rightarrow & \phi (U) \\
& & (z,x) & \mapsto & \phi (\psi _2 ^{-1}(z)\cdot \psi _1 ^{-1}(x))
\end{array}
\]
is a Nash map.
Similarly, there exist open neighborhoods $U_g\subset U$ and $W_1''\subset W_1'$ of the identity and $g$ respectively such that 
\[
\begin{array}{rcrcl}
\psi _1 \circ \cdot \circ (\phi ^{-1},\psi _1 ^{-1})& : & \phi (U_g)\times \psi _1 (W_1'')& \rightarrow & \psi _1 (W_1') \\
& & (x,y) & \mapsto & \psi _1 (\phi ^{-1}(x)\cdot \psi _1 ^{-1}(y))
\end{array}
\]
is a Nash map.
Since semialgebraic sets are closed under projections, we can evaluate the first map at $z=\psi _2 (g^{-1})$ and the second 
at $y=\psi _1 (g)$ to obtain Nash maps again.
Then, composing both maps we deduce that $(U,\phi )$ satisfies $(ii)$ for $g$ of Fact \ref{compatibility0}.
\end{enumerate}
Hence $(U,\phi)$ is under the hypothesis of Fact \ref{compatibility0} and therefore there exists an open neighborhood of the 
identity $V\subset U$ such that $\mathcal{A}_{(V,\phi )}$ is a Nash atlas for $(G,\cdot )$.

\smallskip

Now we check that the identity map from $G$ equipped with $\mathcal{A}$ to $G$ equipped with $\mathcal{A}_{(V,\phi )}$ is a 
locally Nash diffeomorphism.
By Lemma \ref{locally Nash diffeomorphism} it is enough to check that the identity map is both an analytic diffeomorphism and a 
locally Nash map.
Since the identity map is an analytic diffeomorphism between the two analytic groups, it is enough to show that it is a 
locally Nash map.
By definition it suffices to show that for each $g,h\in G$ with $g\in hV$ there exists $(W_1,\psi _1 )\in \mathcal{A}$ with 
$g\in W_1$ and an open neighborhood $W_1'\subset W_1\cap hV$ of $g$ such that ($\psi _1(W_1')$ is semialgebraic and)
\[
\phi _h \circ \psi _1 ^{-1} : \psi _1(W_1') \rightarrow \phi (V): x\mapsto \phi (h^{-1} \psi _1 ^{-1}(x))
\]
is a Nash map.
Let $g$ and $h$ be fixed with $g\in hV$.
Let $(W_2,\psi _2 )\in \mathcal{A}$ be a coordinate neighborhood of $h^{-1}$.
Since $h^{-1}g\in V$ and $\cdot :G\times G\rightarrow G$ is a locally Nash map, when $G$ is equipped with $\mathcal{A}$, there 
exist $(W_1,\psi _1)\in \mathcal{A}$, coordinate neighborhood of $g$, and open neighborhoods $W_2'\subset W_2$ and 
$W_1'\subset W_1$ of $h^{-1}$ and $g$ respectively such that $W_2'\cdot W_1'\subset V$ and
\[
\begin{array}{rcrcl}	
\phi \circ \cdot \circ (\psi _2 ^{-1},\psi _1 ^{-1})& : & (\psi _2 (W_2'), \psi _1(W_1')) & \rightarrow & \phi (V)\\  
& & (x,y) & \mapsto & \phi (\psi _2 ^{-1}(x)\cdot \psi _1 ^{-1}(y))
\end{array}
\]
is a Nash map.
Since semialgebraic sets are closed under projections, we can evaluate the map above at $x=\psi _2 (h^{-1})$ to deduce
that 
\[
\phi _h \circ \psi _1 ^{-1}: \psi _1 (W_1')\rightarrow \phi (V):x \mapsto \phi (h^{-1} \psi _1 ^{-1}(x))
\]
is a Nash map as required.
\end{proof}

The next proposition will provide a sufficient condition for a pure homomorphism of locally Nash groups to be a 
locally Nash homomorphism.
Before proving it, we recall  a result on semialgebraic maps giving a proof different from that in \cite{Fernando_Gamboa_Ruiz}.

\begin{fact}[{\cite[2.4.1]{Fernando_Gamboa_Ruiz}}]\label{Nash cell decomposition}
Let $U$ be an open subset of $\mathbb{R}^n$ and let $f:U\rightarrow \mathbb{R}^m$ be a semialgebraic map.
Then, there exists an open dense subset $V\subset U$ such that $f:V\rightarrow \mathbb{R}^m$ is Nash.
\end{fact}
\begin{proof}
Fix $n$ and $U\subset \mathbb{R}^n$.
We say that $g:U\rightarrow \mathbb{R}$ has complexity $\leq d$ if there is a non-zero polynomial $P$ in $n+1$ variables 
with coefficients in $\mathbb{R}$ of total degree $\leq d$, such that $P(x,g(x))=0$ for all $x\in U$.
We denote by $S^k(U)$ the set of all semialgebraic functions from $U$ to $\mathbb{R}$ such that all its partial derivatives up to order 
$k$ exist and are continuous.
We note that by \cite[Lemma 2.5.2.]{Bochnak_Coste_Roy} for each $i\in \{1,\ldots ,m\}$, there exists a polynomial 
$P_i\in \mathbb{R}[X_1,\ldots ,X_{n+1}]$ such that $P_i(x,f_i(x))=0$, for every $x\in U$.
Hence there exists $C\in \mathbb{N}$ such that all of $f_1,\ldots ,f_m$ have complexity $\leq C$.
By \cite[Theorem 8.10.5]{Bochnak_Coste_Roy} there exists $r=r(n,C)$ such that for every open semialgebraic subset $V$ of 
$\mathbb{R}^n$, every function that belongs to $S^r(V)$ of complexity $\leq C$ is Nash.
Take a $C^r$ cell decomposition of the graph of $f$ (see for example \cite[Ch.7 \textsection 3.3]{van_den_Dries}). Consider the union of all cells of dimension $n$ and let $V$ be its projection over $\mathbb{R}^n$. Then the set $V$ is an open dense subset of $U$ and $f|_{V}$ is Nash.

\end{proof}

\begin{proposition}\label{lochomo} Let $G$ and $H$ be locally Nash groups and let $f:G\rightarrow H$ be a homomorphism of pure 
groups.
Suppose there exist charts $(U,\phi)$ and $(V,\psi)$ of $G$ and $H$ respectively and an open subset $U'\subset U$ such that 
$f(U')\subset V$ and $\psi \circ f \circ \phi^{-1}:\phi(U')\rightarrow \psi(V)$ is a semialgebraic map. 
Then $f$ is a locally Nash homomorphism. 
\end{proposition}
\begin{proof}
Firstly, we note that by Fact \ref{Nash cell decomposition} and restricting $U'$ if necessary, we can assume that 
the map $\psi \circ f \circ \phi^{-1}:\phi(U')\rightarrow \psi(V)$ is Nash.

Now we prove that $f$ is a locally Nash map, provided that $U'\subset U$ and $V$ are neighborhoods of the identity of $G$ and $H$ 
respectively.
By Proposition \ref{compatibility1} we can assume that the locally Nash groups $G$ and $H$ equipped with $\mathcal{A}_{(U ,\phi)}$ and 
$\mathcal{A}_{(V ,\psi)}$ are locally Nash isomorphic to the original structures.
Let $g\in G$.
We have that $(gU,\phi_g)$ and $(f(g)V,\psi_{f(g)})$ are charts of $G$ and $H$ respectively with $g\in gU'\subset gU$ and 
$f(g)\in f(g)V$.
By Proposition \ref{characterization of locally Nash maps}.$(3)$ it would be enough to show that the map
\[
\psi_{f(g)}\circ f\circ \phi^{-1}_g:\phi(U') \rightarrow \psi(V) 
\]
is Nash.
The latter is true since 
\[
(\psi_{f(g)}\circ f\circ \phi^{-1}_g)(x)=(\psi_{f(g)}\circ f)(g\phi^{-1}(x))=
\]
\[
=\psi_{f(g)}(f(g)f(\phi^{-1}(x)))=\psi(f(g)^{-1}f(g)f(\phi^{-1}(x)))=\psi(f(\phi^{-1}(x))) 
\]
and $\psi \circ f \circ \phi^{-1}$ is Nash.

It remains to prove that we can assume that the relevant open sets can be taken neighborhoods of the identity.
Fix $g\in U'$.
Since the group operation of $G$ is a locally Nash map, there exists a chart $(U_0,\phi_0)$ of $G$ 
and an open neighborhood of the identity $U'_0\subset U_0$ such that $gU'_0\subset U'$ and the map
\[
L_g:\phi_0(U'_0)\rightarrow \phi(U'):x\mapsto \phi(g\phi_0^{-1}(x))
\]
is a Nash map, in particular semialgebraic.
Similarly, there exists a chart $(V_0,\psi_0)$ of the identity of $H$ and an open subset $V'\ni f(g)$ of $V$ such that 
$f(g)^{-1}V'\subset V_0$ and 
\[
L_{f(g)^{-1}}:\psi(V')\rightarrow \psi_0(V_0):x\mapsto \psi_0(f(g)^{-1}\psi^{-1}(x))
\]
is a semialgebraic map.
By continuity and since $(\psi \circ f \circ \phi^{-1}\circ L_g)(\phi_0(e))=\psi(f(g))$, we can take $U'_0$ small enough so that 
\[
(\psi \circ f \circ \phi^{-1}\circ L_g)(\phi _0(U'_0))\subset\psi(V').
\]
In particular the composition
\[
\begin{array}{rcccl}
L_{f(g)^{-1}}\circ \psi \circ f \circ \phi^{-1}\circ L_g & : & \phi_0(U'_0) & \rightarrow & \psi_0(V_0) \\
& & x & \mapsto & \psi_0(f(\phi^{-1}_0(x)))
\end{array}
\]
is a semialgebraic map, as required.
\end{proof}

Next we will characterize those isomorphisms of pure groups which are isomorphisms of locally Nash groups.

\begin{proposition}\label{compatibilityAut}
Let $G$ and $H$ be locally Nash groups equipped with atlases $\mathcal{A}$ and $\mathcal{B}$ respectively.
Then, a continuous isomorphism $\alpha :G \rightarrow H$ is an isomorphism of locally Nash groups if and only if there exist 
(for all) charts of the identity $(U,\phi)\in\mathcal{A}$ and $(V, \psi) \in \mathcal{B}$ with an open neighborhood of the identity  
$W\subset  U \cap \alpha^{-1}(V)$ such that $\psi\circ \alpha$ is algebraic over $\mathbb{R}(\phi)$ on $W$. 
\end{proposition}
\begin{proof}
We first prove the right to left implication.
Fix $i\in\{1,\ldots,n\}$.
By hypothesis $\psi_i\circ \alpha$ is algebraic over $\mathbb{R}(\phi)$ on $W$ and therefore, since $\phi$ is a diffeomorphism,
$\psi_i \circ \alpha \circ \phi^{-1}$ is algebraic over $\mathbb{R}(\text{id})$ on $\phi(W)$. 
Hence there exists a polynomial $P\in \mathbb{R}[x][Y]$ such that $P(x,(\psi_i \circ \alpha \circ \phi ^{-1})(x)) = 0$ for all 
$x \in \phi(W)$.
Without loss of generality, we can assume that $\phi(W)$ is semialgebraic.
Then, by the proof of \cite[Proposition 8.1.8]{Bochnak_Coste_Roy} and since $\psi_i \circ \alpha \circ \phi ^{-1}$ is continuous, 
we obtain that each coordinate function $\psi_i \circ \alpha \circ \phi^{-1}$ is a semialgebraic function on $\phi(W)$.
By Proposition \ref{lochomo} we deduce that $\alpha$ is a locally Nash map.
Moreover, we also have that the inverse of the above map,
\[
\phi \circ \alpha^{-1} \circ \psi^{-1}:\psi(\alpha(W))\rightarrow \phi(W)\subset \phi(U), 
\]
is  semialgebraic and therefore, again by Proposition \ref{lochomo}, we deduce that $\alpha^{-1}$ is a locally Nash map.
Thus $\alpha$ is a locally Nash isomorphism.  

Now, we show the left to right implication.
Fix charts of the identity $(U,\phi )\in \mathcal{A}$ and $(V,\psi )\in \mathcal{B}$.
Since $\alpha$ is a locally Nash map, by Proposition \ref{characterization of locally Nash maps}.$(2)$ there exists an open 
neighborhood of the identity $W\subset U\cap \alpha ^{-1}(V)$ such that
\[
\psi \circ \alpha \circ \phi ^{-1} : \phi (W) \rightarrow \psi (V): x\mapsto \psi (\alpha ( \phi ^{-1}(x)))
\]
is a Nash map.
So $\psi \circ \alpha \circ \phi ^{-1}$ is algebraic over $\mathbb{R}(\text{id})$ on $\phi(W)$ and hence $\psi \circ \alpha$ 
is algebraic over $\mathbb{R}(\phi)$ on $W$ as required.
\end{proof}

Proposition \ref{compatibilityAut} leads to an immediate corollary:

\begin{corollary}\label{compatibilityGerm}
Let $(G,\cdot )$ equipped with a Nash atlas $\mathcal{A}$ be a locally Nash group.
Let $(U,\phi )$ and $(V,\psi)$ be charts of the identity of $\mathcal{A}$.
If $(G,\cdot,\phi |_U)$ and $(G,\cdot,\psi |_V)$ are locally Nash groups then they are locally Nash isomorphic.
\end{corollary}
\begin{proof}
Since $(U,\phi )$ and $(V,\psi)$ are charts that are Nash compatible,
\[
\psi \circ \phi ^{-1}: \phi (U\cap V)\rightarrow \psi (U\cap V) : x\mapsto \psi \circ \phi ^{-1}(x)
\]
is semialgebraic.
So $\psi$ is algebraic over $\mathbb{R}(\phi)$ on $U\cap V$. 
Now apply Proposition \ref{compatibilityAut} taking $\alpha$ as the identity map.
\end{proof}

As a special case of Corollary \ref{compatibilityGerm}, we have that if there exists neighborhoods of the identity $U$ and $V$ such that 
$(G,\cdot,\phi |_U)$ and $(G,\cdot,\phi |_V)$ are locally Nash groups then both are locally Nash isomorphic.

\section{Algebraic Addition Theorems.}\label{AAT}

In this section we review the principal properties of algebraic addition theorems which can be found in the literature 
and we prove some new ones, in particular Theorem \ref{T2}.
Since most of the main references on this concept are outdated, we include proofs with modern notation.
We will work with power series instead of with germs of analytic functions.
There are two good reasons to do this.
The first one is that most of the classical sources about algebraic addition theorems follow this line.
The second one is that working in an algebraic context we can remark that some calculations are formal. 
We do not want to think in terms of germs and meromorphic functions yet, this will be done in the next section, 
where we will extend our results to meromorphic and real meromorphic functions.

Firstly, we introduce the notation for the power series.
Let $\mathbb{K}$ be $\mathbb{C}$ or $\mathbb{R}$.
Let $A_{\mathbb{K},n}$ be the ring of all power series in $n$ variables with coefficients in $\mathbb{K}$ that are convergent 
in a neighborhood of the origin.
We recall that $A_{\mathbb{K},n}$ is an integral domain.
Let $M_{\mathbb{K},n}$ be the quotient field of $A_{\mathbb{K},n}$.
For each $\epsilon >0$ we denote $U_{\mathbb{K},n}(\epsilon)$ the open ball $\{  k\in \mathbb{K}^n:\|  k\| <\epsilon \}$.
Since we will only consider convergence over open subsets of $\mathbb{C}^n$, we denote $U_n(\epsilon )$ the open ball 
$U_{\mathbb{C},n}(\epsilon)$.
Thus, we say $(\phi _1,\ldots ,\phi _m) \in M_{\mathbb{K},n}^m$ is {\itshape convergent} in $U_n(\epsilon )$ if each 
$\phi _1,\ldots ,\phi _m$ is the quotient of two power series convergent in $U_n(\epsilon )$.

Let us recall the relation between power series and analytic functions
(see \cite{Gunning_Rossi} and \cite{Narashiman} for basic notions of complex and real analytic geometry respectively).
Let $U\subset \mathbb{K}^n$ be an open connected neighborhood of $0$. 
We denote by $\mathcal{O}_{\mathbb{K},n}(U)$ the ring of all analytic functions in $U$ and by $\mathcal{O}_{\mathbb{K},n}$ 
the ring of germs of analytic functions at $0$. For each $f$ in $\mathcal{O}_{\mathbb{K},n}(U)$ or $\mathcal{O}_{\mathbb{K},n}$ we denote by ${^t f}$ its Taylor power series expansion at $0$. The map
\[
^a : A_{\mathbb{K},n}\rightarrow \mathcal{O}_{\mathbb{K},n}: \phi \mapsto {^a \phi}
\]
that assigns to each $\phi$ the germ of the analytic function
\[
^a \phi :U_\phi \rightarrow \mathbb{K}:  k \mapsto \phi (k) 
\]
where $U_\phi \subset \mathbb{K}^n$ is an open neighborhood of $0$ where $\phi$ converges, is an isomorphism of rings whose inverse is given by the Taylor power series expansions at $0$.

On the other hand, by means of the identity principle for analytic functions, both $\mathcal{O}_{\mathbb{K},n}(U)$ and $\mathcal{O}_{\mathbb{K},n}$ are integral domains. We denote 
by $\mathcal{M}_{\mathbb{K},n}(U)$ and $\mathcal{M}_{\mathbb{K},n}$ their respectives quotient fields. The maps $^a$ and $^t$ are naturally defined for these quotients fields and give us also an isomorphism of $M_{\mathbb{K},n}$  and $\mathcal{M}_{\mathbb{K},n}$. Similarly, we define ${^ a}$ and ${^ t}$ for tuples.

\begin{remark}\label{rmkPoincare}\emph{A meromorphic function on $U$ is a global section of the sheaf over $U$ whose stalk in $x\in U$ is the quotient field of the germs of analytic functions in $x$. In other words, a meromorphic function on $U$ is given by an open covering $\{U_j\}_{j\in J}$ of $U$ and a collection of analytic functions $h_j,g_j:U_j\rightarrow \mathbb{C}$ such that 
$$g_j \cdot h_{\ell}=g_{\ell} \cdot h_j \qquad \text{ in } U_j\cap U_{\ell}.$$ 
Although clearly the elements of $\mathcal{M}_{\mathbb{K},n}(U)$ are meromorphic functions, the converse is not necessarily true. The problem of determining whether or not the converse holds for a certain $U$ is known as the Poincar\'e problem. For example, it holds if $U=\mathbb{C}^n$ (see \cite[Ch. VIII, \textsection B, Corollary 10]{Gunning_Rossi}).
}
\end{remark}

Now we give a more precise formulation of algebraic addition theorem.
\begin{notation}\label{notation-AAT}
\emph{ Let $\phi :=(\phi _1,\ldots ,\phi _m)\in M_{\mathbb{K},n}^m$ be convergent in $U_n(\epsilon )$, let 
$k\in U_{\mathbb{K},n}(\epsilon )$ and let $(u,v):=(u_1,\ldots ,u_n,v_1,\ldots ,v_n)$ be $2n$ variables. 
We will use the following notation:
\begin{enumerate}
\item[$(1)$] $\phi _{(u,v)}:=\big(\phi _1( u),\ldots ,\phi _m( u), \phi _1( v),\ldots ,\phi _m( v)\big)
\in M_{\mathbb{K},2n}^{2m}$.
\item[$(2)$] $\phi _{ u+ v}:=\big(\phi _1( u+ v),\ldots ,\phi _m( u+ v)\big)
\in M_{\mathbb{K},2n}^m$.
\item[$(3)$] $\phi _{u+k}:=\big(\phi _1( u+ k),\ldots ,\phi _m( u+ k)\big)
\in M_{\mathbb{K},n}^m$.
\end{enumerate}
}
\end{notation}
Given $\phi \in M_{\mathbb{K},\ell}^{n}$ and $\psi \in M_{\mathbb{K},\ell}^{m}$ we say that $\phi$ is {\itshape algebraic} over 
$\mathbb{K}(\psi):=\mathbb{K}(\psi _1 , \ldots ,\psi _m)$ if $\phi _1,\ldots ,\phi _n$ are algebraic over $\mathbb{K}(\psi)$.
\begin{definition}
\emph{
We say $\phi \in M_{\mathbb{K},n}^n$ admits an \emph{algebraic addition theorem (AAT)} if $\phi _1,\ldots ,\phi _n$ are 
algebraically independent over $\mathbb{K}$ and $\phi _{u+v}$ is algebraic over $\mathbb{K}(\phi _{(u,v)})$.
}
\end{definition}

The rest of the section is divided as follows.
Firstly, we will adapt some technical lemmas of \cite{Bochner_Martin} and \cite{Siegel} to our context to prove some separated rationality result (Proposition \ref{S-BM-alg}), 
these results will be needed for proving Theorem \ref{T2}.
Secondly, we will prove some properties of elements of $M_{\mathbb{K},n}$ admitting an AAT and Theorem \ref{T2}.
Finally, we will prove some properties of differentials that will be needed for proving Theorem \ref{T1} in 
Section \ref{meromorphic maps}, when we consider AAT related to charts of the identity of locally Nash groups.

\medskip

\subsection{A separate rationality result.} We begin with some technical lemmas. Let $u=(u_1,\ldots,u_p)$ and $v=(u_1,\ldots,u_q)$ be variables and let $\phi\in M_{\mathbb{K},p+q}^m$ be convergent in $U_{p+q}(\epsilon)$, that is, $\phi(u,v)=\frac{\alpha(u,v)}{\beta(u,v)}$ for $\alpha,\beta \in \mathbb{A}_{\mathbb{K},p+q}$ convergent in $U_{p+q}(\epsilon)$. Given a point $k\in \mathbb{K}^q$ we write $\phi (u,k)\in M_{\mathbb{K},p}$ if $\beta(u,k)\neq 0$.

\begin{lemma}\label{pole} Let $\phi :=(\phi _1,\ldots ,\phi _m)\in M_{\mathbb{K},n}^m$ be convergent in $U_n(\epsilon)$.
Let $p,q\in \mathbb{N}$ such that $p+q=n$.
\begin{enumerate}
\item[$(1)$]\label{analytic} There exists an open dense subset $U$ of $U_{\mathbb{K},n}(\epsilon)$ such that
\[
^a \phi :U\rightarrow \mathbb{K}^m: k \mapsto \phi (k) 
\]
is an analytic function.
\item[$(2)$] There exists an open dense subset $V$ of $U_{\mathbb{K},q}(\epsilon)$ such that
\[
V\subset \{ k\in U_{\mathbb{K},q}(\epsilon): \phi (u,k)\in M_{\mathbb{K},p}^m\}.
\]
\item[$(3)$] If there exists an open subset $W$ of $U_{\mathbb{K},q}(\epsilon)$ such that
\[
W\subset \{ a\in U_{\mathbb{K},q}(\epsilon) : \phi (u,a)\in M_{\mathbb{K},p} \text{ and } \phi (u,a)=0\}
\]
then $\phi =0$.
\end{enumerate}
\end{lemma}
\begin{proof}
For each $i\in \{1,\ldots ,m\}$ let $\alpha _i,\alpha _{m+i}\in A_{\mathbb{K},n}$, $\alpha _{m+i}\neq 0$, such that 
$\phi _i=\frac{\alpha _i}{\alpha _{m+i}}$.
For each $i\in \{1,\ldots ,2m\}$ let $^a\alpha _i:U_n(\epsilon)\rightarrow \mathbb{K}: k\mapsto \alpha _i( k)$.

For the first property we note that since $^a\alpha _{m+1},\ldots ,^a\alpha _{2m}$ are analytic in $U_n(\epsilon)$ 
and not identically zero, the set
\[
U:=\{ k\in U_{\mathbb{K},n}(\epsilon ): \alpha _{m+1}(k)\cdot \ldots \cdot \alpha _{2m}(k) \neq 0\}
\]
is an open dense subset of $U_{\mathbb{K},n}(\epsilon)$ by the identity principle.

For the second property we may project and take the open set $V:=\pi (U)$.

For the third property we note that since $^a\alpha _1,\ldots ,^a\alpha _m$ are analytic in $U$ and identically zero in 
$\{ (a,b)\in U : b\in W\}$, they are identically zero in an open subset of $U$.
So $\alpha _1=\ldots =\alpha _m=0$.
\end{proof}

We anticipate that in the next lemma the hypothesis that $f( u, v)\in M_{\mathbb{K},2n}$ is convergent in 
$\{(a,b)\in \mathbb{C}^{2n}: a+b\in U_n(\epsilon )\}$ can be weakened to be convergent in some open subset 
of $\mathbb{C}^{2n}$ containing $U_n(\epsilon )\times \{ 0\}$.

\begin{lemma}\label{core}
Let $\epsilon >0$.
Let $\phi \in M_{\mathbb{K},n}^m$ be convergent in $U_n(\epsilon )$ and $\phi _1,\ldots ,\phi _m$ 
algebraically independent over $\mathbb{K}$.
Let $f(u,v)\in M_{\mathbb{K},2n}$ be convergent in $\{(a,b)\in \mathbb{C}^{2n}: a+b\in U_n(\epsilon )\}$ 
and $f(u,k)\in M_{\mathbb{K},n}$ for all $k\in U_{\mathbb{K},n}(\epsilon )$.
If $f(u,v)$ is algebraic over $\mathbb{K}(\phi _{(u,v)})$ then $f(u,k)$ is algebraic over 
$\mathbb{K}(\phi (u))$ for each $k\in U_{\mathbb{K},n}(\epsilon )$.
Furthermore, there exist $N\in \mathbb{N}$ and $h\leq N$ such that for each $k\in U_{\mathbb{K},n}(\epsilon )$, 
the minimal polynomial of $f(u,k)$ over $\mathbb{K}(\phi (u))$ can be written in the form
\[
Y^h+\sum _{i=0}^{h-1} \frac{R_i(\phi )}{S_i(\phi )} Y^i \quad R_i,S_i \in \mathbb{K}[X_1,\ldots ,X_m]^{\leq N}, S_i\neq 0,
\]
where $\mathbb{K}[X_1,\ldots ,X_m]^{\leq N}$ denotes the polynomials of $\mathbb{K}[X_1,\ldots ,X_m]$ whose degree in each of 
the variables $X_1,\ldots ,X_m$ is bounded by $N$.
\end{lemma}
\begin{proof}
We could try to evaluate the minimal polynomial of $f(u,v)$ over $\mathbb{K}(\phi _{(u,v)})$ at $k$ but we may have problems if any 
of the denominators becomes the zero polynomial in $\mathbb{K}(\phi (u))$.
So we are going to modify the original polynomial while keeping the original degree.

Since $f(u,v)$ is algebraic over $\mathbb{K}(\phi _{(u,v)})$ and $\phi_1,\ldots ,\phi _m$ algebraically independent over 
$\mathbb{K}$, there exists $P\in \mathbb{K}[X_1,\ldots ,X_{2m}][Y]$ such that $P(\phi _{(u,v)};Y)\neq 0$ and 
$P(\phi _{(u,v)};f (u,v))= 0$.
Hence, there exists $N\in \mathbb{N}$ such that
\[
P(X_1,\ldots ,X_{2m};Y)=\sum _{j, \mu ,\nu \leq N}  a_{j,\mu ,\nu } \ X_1^{\mu _1}\ldots X_m^{\mu _m}
X_{m+1}^{\nu _1}\ldots X_{2m}^{\nu _m}Y^j,
\]
with $a_{j,\mu ,\nu }\in \mathbb{K}$ and where for each $\delta \in \mathbb{N}^m$, $\delta \leq N$ denotes 
$\delta _1\leq N$,\ldots ,$\delta _m\leq N$.
We will prove that this $N$ is the required one in the statement of the lemma.
Firstly we prove some claims.

\begin{claim}\label{claim1} There exists an open dense subset $U$ of $U_{\mathbb{K},n}(\epsilon )$ such that for each 
$k\in U$, $P(X_1,\ldots ,X_m,\phi (k);Y)$ is a non-zero polynomial of $\mathbb{K}[X_1,\ldots ,X_m][Y]$.
\end{claim}
\noindent {\it Proof of Claim \ref{claim1}.}
By Lemma \ref{pole}.$(1)$ there exists an open dense subset $W\subset U_{\mathbb{K},n}(\epsilon )$ such that
\[
W\subset \{  k\in U_{\mathbb{K},n}(\epsilon ): \phi ( k)\in \mathbb{K}^m \}
\]
and $^a\phi :W\rightarrow \mathbb{K}^m:  k\mapsto \phi ( k)$ is analytic.
Let 
\[
U:=\{  k\in W: P(X_1,\ldots ,X_m,\phi ( k);Y) \neq 0\}. 
\]
Since $W$ is an open dense subset of $U_{\mathbb{K},n}(\epsilon )$, to prove the claim it is enough to show that 
$W\setminus U$ is closed and nowhere dense in $W$.
Clearly $W\setminus U$ is closed in $W$ since $^a\phi$ is continuous in $W$.
For the density, we note that if $W\setminus U$ contains an open subset of $W$ then
\[
\{  k \in U_{\mathbb{K},n}(\epsilon ): P(\phi ( u),\phi ( k);Y)\in M_{\mathbb{K},n+1} \text{ and } 
P(\phi ( u),\phi ( k);Y)=0\}
\]
contains an open subset of $U_{\mathbb{K},n}(\epsilon )$ and therefore $P(\phi _{(u,v)};Y)=0$ by Lemma \ref{pole}.$(3)$.
This finishes the proof of Claim \ref{claim1}.

\begin{claim}\label{claim2} For each $ k \in U_{\mathbb{K},n}(\epsilon )$ there exists 
$Q_{ k}\in \mathbb{K}[X_1,\ldots ,X_m][Y]$ such that $Q_{ k} (\phi ( u);Y)\neq 0$, 
$Q_{ k} (\phi ( u); f ( u, k))= 0$ and $Q_{ k}$ is a sum of monomials of the form
\[
a \, X_1^{\mu _1}\ldots X_m^{\mu _m}Y ^j,\quad a\in \mathbb{K},\ \mu _1,\ldots ,\mu _m\leq N,\ j\leq N.
\]
\end{claim}
\noindent {\it Proof of Claim \ref{claim2}.}
We follow the proof of \cite[Chap. IX. \textsection 5. Theorem 5]{Bochner_Martin}.
For each $k\in W$, where $W$ is as in the proof of Claim \ref{claim1}, let
\[
P_{ k}(X_1,\ldots ,X_m;Y)=\sum _{j,\mu \leq N} d_{j,\mu , k} \ X_1^{\mu _1}\ldots X_m^{\mu _m}Y ^j
\]
denote the polynomial $P(X_1,\ldots ,X_m,\phi ( k);Y)$.
By Claim \ref{claim1} there exists an open dense subset $U$ of $U_{\mathbb{K},n}(\epsilon )$ where $P_{ k}\neq 0$ 
for all $ k\in U$.
For each $k\in U$, we define
\[
E(P_{ k}):= \sum _{j,\mu \leq N} \| d_{j,\mu ,k} \| ^2.
\]
We note that $E(P_{k})>0$ for all $ k\in U$.
For each $k\in U$, let 
\[
Q_{k}(X_1,\ldots ,X_m;Y):=\sum _{j,\mu \leq N} b_{j,\mu , k} \ X_1^{\mu _1}\ldots X_m^{\mu _m}Y ^j,
\]
where
\[
b_{j,\mu ,k}:=\frac{d_{j,\mu , k}}{\sqrt{E(P_{ k})}}. 
\]
So, we have, for each $k\in U$, $Q_{k}(\phi (u) ; Y)\neq 0$, $Q_{k}(\phi (u) ; f(u,k))=0$ and $E(Q_{k})=1$.
We define
\[
\vec{v}( k):=(b_{j,\mu , k})_{j,\mu \leq N}\in \{ z\in \mathbb{K}^{(N+1)^{(m+1)}}: \|z\|=1 \}.
\]
Take $ k\in U_{\mathbb{K},n}(\epsilon)\setminus U$.
Since $U$ is an open dense subset of $U_{\mathbb{K},n}(\epsilon )$, there exists a Cauchy sequence 
$\{k_r\}_{r \in \mathbb{N}}\subset U$ that converges to $ k$.
For each $k_r$, the identity $Q_{ k_r}(\phi ( u) ;f( u, k_r))=0$ holds, therefore
\[
\sum _{j,\mu \leq N} b_{j,\mu , k _r}\phi _1( u) ^{\mu _1}\ldots \phi _m( u)^{\mu _m} f ( u, k _r)^j=0.
\]
By hypothesis there are $\alpha ,\beta \in A_{\mathbb{K},2n}$, $\beta \neq 0$, convergent in 
$\{( a, b)\in \mathbb{K}^{2n}:  a+ b\in U_n(\epsilon )\}$, such that $f( u, v) =\frac{\alpha ( u, v)}{\beta ( u, v)}$.
In particular
\[
\sum _{j,\mu \leq N} b_{j,\mu , k _r}\phi _1( u) ^{\mu _1}\ldots \phi _m( u)^{\mu _m} 
\alpha ( u, k _r)^j\beta ( u, k _r)^{N-j}=0.\tag{$\ast $}\label{main}
\]
Since $\{ z\in \mathbb{K}^{(N+1)^{(m+1)}}: \|  z\|=1 \}$ is compact, taking an adequate subsequence we can assume that the limit
of the sequence $\{\vec{v}( k_r)\}_{r\in \mathbb{N}}$ exists.
For each $j,\mu \leq N$ we define
\[
b_{j,\mu , k }:=\lim _{s \rightarrow \infty} \ b_{j,\mu , k_r}.
\]
Since $\alpha$ and $\beta$ is continuous, when $r$ tends to infinity equation (\ref{main}) becomes
\[
\sum _{j,\mu \leq N} b_{j,\mu , k}\phi _1( u) ^{\mu _1}\ldots \phi _m( u)^{\mu _m} 
\alpha ( u, k)^j\beta ( u, k)^{N-j}=0.
\]
So dividing by $\beta (u,k)^N$, we also have
\[
\sum _{j,\mu \leq N} b_{j,\mu , k}\phi _1(u) ^{\mu _1}\ldots \phi _m(u)^{\mu _m}  f( u, k)^j=0
\]
and hence the polynomial
\[
Q_k(X_1,\ldots ,X_{m+1};Y):= \sum _{j,\mu \leq N} b_{j,\mu , k}X_1 ^{\mu _1}\ldots X_m^{\mu _m} Y^j
\]
satisfies $Q_k(\phi ( u) ,f( u, k))=0$.
We note that $E(Q_k)=\lim _{r\rightarrow \infty} E(Q_{ k_r}) =1$, so $Q_k\neq 0$.
Since $\phi _1,\ldots ,\phi _m$ are algebraically independent over $\mathbb{K}$ and $Q_k(X_1,\ldots ,X_m,Y)\neq 0$ then 
$Q_k(\phi ( u), Y)\neq 0$.
This finishes the proof of Claim \ref{claim2}.

\smallskip

Claim \ref{claim2} implies that $f( u, k)$ is algebraic over $\mathbb{K}(\phi ( u))$ for all 
$ k\in U_{\mathbb{K},n}(\epsilon )$.
It remains to check the conditions on $N$ and on the minimal polynomials.
Fix $k\in U_{\mathbb{K},n}(\epsilon )$.
Let $A(Y):=Q_{ k}(\phi ( u);Y)$, where $Q_{ k}$ is the polynomial of Claim \ref{claim2}.
By definition of $Q_{k}$, in the proof of Claim \ref{claim2}, we have
\[
A(Y)=A_dY^d+\sum _{j=0}^{d-1}A_jY^j,\quad A_0,\ldots ,A_d\in \mathbb{K}[\phi ( u)],\ A_d\neq 0,\ d\leq N
\]
where each of $A_0$,\ldots ,$A_d$ is a sum of monomials of the form
\[
a\, \phi _1( u)^{\mu _1}\ldots \phi _m( u)^{\mu _m},\quad a\in \mathbb{K},\ 0\leq \mu _1,\ldots \mu _m\leq N.
\]
Let 
\[
B(Y)=Y^e+\frac{\sum _{j=0}^{e-1}B_jY^j}{B_e},\quad B_0,\ldots ,B_e\in \mathbb{K}[\phi ( u)]
\]
be the minimal polynomial of $f( u, k)$ over $\mathbb{K}(\phi ( u))$.
Since $f( u, k)$ is both a root of $A(Y)$ and of $B(Y)$, there exists 
\[
C(Y)=C_\ell Y^\ell+\sum _{j=0}^{\ell-1}C_jY^j,\quad C_0,\ldots ,C_\ell\in \mathbb{K}[\phi ( u)]
\]
such that $A(Y)=B(Y)\,C(Y)$.
Therefore 
\[
B_e\, \left(A_dY^d+\sum _{j=0}^{d-1} A_jY^j\right) = \left(B_eY^e+\sum _{j=0}^{e-1} B_jY^j\right) 
\left(C_\ell Y^\ell+\sum _{j=0}^{\ell -1} C_jY^j\right).
\]
We note that $\mathbb{K}[\phi ( u)]\cong \mathbb{K}[X_1,\ldots ,X_m]$ because $\phi _1,\ldots ,\phi _m$ are algebraically 
independent over $\mathbb{K}$.
Since $\mathbb{K}[\phi ( u)][Y]$ is an UFD and $B(Y)$ is irreducible,
\[
B_eY^e+\sum _{j=0}^{e-1} B_jY^j  \text{  divides  } A_dY^d+\sum _{j=0}^{d-1} A_jY^j \text{  in  } \mathbb{K}[\phi ( u)][Y].
\]
This implies that each of $B_0$,\ldots ,$B_e$ is a sum of monomials of the form
\[
a\, \phi _1( u)^{\mu _1}\ldots \phi _m( u)^{\mu _m},\quad a\in \mathbb{K},\ 0\leq \mu _1,\ldots \mu _m\leq N.
\]
This proves the statement for $f(u,k)$.
Since $k$ was fixed, we are done.
\end{proof}

In Proposition \ref{S-BM-alg} we will adapt to our setting a well known result that says that a complex analytic function
$f(u,v)$ that is rational in $u$, for each fixed $v$, and rational in $v$, for each fixed $u$, belongs to $\mathbb{C}(u,v)$; 
see \cite[Chap. IX. \textsection 5. Theorem 5]{Bochner_Martin} for details.
We begin with a lemma based on \cite[Chap. 5. \textsection 13. Theorem 1]{Siegel} that we include for completeness.

\begin{lemma}\label{S-BM} Let $\Psi :=(\psi _0,\ldots ,\psi _n)\in M_{\mathbb{K},n}^{n+1}$ be convergent in $U_n(\epsilon)$. 
Let $\psi _1,\ldots ,\psi _n$ be algebraically independent over $\mathbb{K}$ and let $\psi _0$ be algebraic over 
$\mathbb{K}(\psi _1,\ldots ,\psi _n)$. 
Let $\psi :=(\psi _1,\ldots ,\psi _n)$ and let $h$ be the degree of the minimal polynomial of $\psi _0$ over $\mathbb{K}(\psi )$.
Let $f(u,v)\ \in M_{\mathbb{K},2n}$ be convergent in $U_{2n}(\epsilon )$.
If there exists $N\in \mathbb{N}$ such that for each $k\in U_{\mathbb{K},n}(\epsilon )$ we have that 
$f(k,v)\in \mathbb{K}(\Psi (v))$ and there are $R_i,S_i \in \mathbb{K}[X_1,\ldots ,X_n]^{\leq N}$, $S_i\neq 0$, such that
\[
f(u,k) = \sum _{i=0}^{h-1} \frac{R_i(\psi (u))}{S_i(\psi (u))}\psi _0(u)^i
\] then $f(u,v)\in \mathbb{K}(\Psi _{(u,v)})$.
\end{lemma}
\begin{proof}
We denote by $H_1( u, v),\ldots ,H_m( u, v)$ the monomials 
\[
\psi _1( u)^{\alpha _1}\ldots \psi _n( u)^{\alpha _n}f( u, v), 
\quad \psi _0( u)^{\alpha _0}\psi _1( u)^{\alpha _1}\ldots \psi _n( u)^{\alpha _n},
\]
with $0\leq \alpha _0\leq h-1$ and $0\leq \alpha _1,\ldots ,\alpha _n \leq hN$ (so $m:=(h+1)(hN+1)^n$).

Firstly we show that if there exists an equation of the form
\[
\xi _1 ( v)H_1( u, v)+\ldots +\xi _m( v)H_m( u, v)=0, \tag{$1$}\label{E2}
\]
where $\xi _1( v),\ldots ,\xi _m( v)\in \mathbb{K}(\Psi ( v))$ and not all of them are $0$ then we are done.
We may assume that there exists $\ell<m$ such that $f( u, v)$ appears in the monomial $H_i( u, v)$ if and only if $i\leq \ell$.
If $\xi _i( v)\neq 0$ for some $i\leq \ell$ then we are done.
Indeed, since $\psi _1,\ldots ,\psi _n$ are algebraically independent over $\mathbb{K}$, we can solve equation (\ref{E2}) 
with respect to $f( u, v)$ to deduce that $f( u, v)\in \mathbb{K}(\Psi _{(u,v)})$.
So it is enough to show that $\xi _i( v)\neq 0$ for some $i\leq \ell$.
Suppose for a contradiction that $\xi _i( v)=0$ for all $i\leq \ell$. 
Since not all $\xi _1( v),\ldots ,\xi _m( v)$ are $0$ we may assume that $\xi _{\ell+1}( v)\neq 0$.
By Lemma \ref{pole}.$(1)$ there exists $ k\in U_{\mathbb{K},n}(\epsilon )$ such that 
$\xi _1( k),\ldots ,\xi _m( k)\in \mathbb{K}$ and $\xi _{\ell+1}( k)\neq 0$.
We note that $H_{\ell+1}(u,v),\ldots ,H_m(u,v) \in \mathbb{K}(\Psi (u))$.
Since they do not depend on $v$, we denote them by $H_i(u)$.
Evaluating equation (\ref{E2}) at $v=k$ we obtain that
\[
\xi _{\ell+1} ( k)H_{\ell+1}( u)+\ldots +\xi _m( k)H_m( u)=0
\]
where $\xi _{\ell+1}( k)\neq 0$.
Since the degree of each $H_i( u)$ in the variable $\psi _0( u)$ is smaller than that of the minimal polynomial of 
$\psi _0$
over $\mathbb{K}(\psi )$, we must have $\xi _{\ell+1}( k)=\ldots =\xi _m( k)=0$, a contradiction.

We now show how to obtain equation (\ref{E2}).
If $f(u,k)=0$ for each $k\in U_{\mathbb{K},n}(\epsilon )$ then $f(u,v)=0$ by Lemma \ref{pole}.$(3)$ and 
there is nothing to prove.
So we may assume that there exists $ k \in U_{\mathbb{K},n}(\epsilon )$ such that $f( u, k)\neq 0$.
By hypothesis for this $k$ there are $R_i,S_i \in \mathbb{K}[X_1,\ldots ,X_n]^{\leq N}$, $S_i\neq 0$, such that
\[
f( u, k) = \sum _{i=0}^{h-1} \frac{R_i(\psi ( u))}{S_i(\psi ( u))}\psi _0( u)^i.
\]
Clearing denominators we get 
\[
S(\psi ( u))f( u, k) = \sum _{i=0}^{h-1} R'_i(\psi ( u))\psi _0( u)^i\tag{$2$}\label{E1}
\]
where $R'_i,S\in \mathbb{K}[X_1,\ldots ,X_n]^{\leq hN}$ and $S\neq 0$.
We also recall that $h$ and $N$ do not depend on $k$.
Now we follow \cite[Chap. IX. \textsection 5. Lemma 6]{Bochner_Martin}.
Let $u_{(1)},\ldots , u_{(m)}$ be independent $n$-tuples of variables and let $D( v, u_{(1)},\ldots , u_{(m)})$ be 
the determinant of 
\[
H( v, u_{(1)},\ldots , u_{(m)}):=
\left[\begin{matrix}
H_1( u_{(1)}, v) & H_1( u_{(2)}, v) & \ldots & H_1( u_{(m)}, v)\\
H_2( u_{(1)}, v) & H_2( u_{(2)}, v) & \ldots & H_2( u_{(m)}, v)\\
\vdots & \vdots & \ddots & \vdots \\
H_m( u_{(1)}, v) & H_m( u_{(2)}, v) & \ldots & H_m( u_{(m)}, v)\\
\end{matrix}\right].
\]
By equation (\ref{E1}), for each $k\in U_{\mathbb{K},n}(\epsilon )$ the monomials 
$H_1(u,k)$,\ldots ,$H_m(u,k)$ are linearly dependent over $\mathbb{K}$.
Since
\[
\{  k\in U_{\mathbb{K},n}(\epsilon ): D( k, u_{(1)},\ldots , u_{(m)})\in M_{\mathbb{K},mn} \text{ and } 
D( k, u_{(1)},\ldots , u_{(m)})=0\}
\]
is $U_{\mathbb{K},n}(\epsilon )$, $D=0$ by Lemma \ref{pole}.$(3)$.
Expanding the determinant of $H$ with respect to its last column, replacing $ u_{(m)}$ by $ u$ and denoting 
$(u_{(1)},\ldots ,u_{(m-1)})$ by $u_{(*)}$, we obtain a new equation of the form
\[
\chi _1 (v,u_{(*)})H_1(u,v)+\ldots +\chi _m(v,u_{(*)})H_m(u,v)=0,
\]
where
\[
\chi _1(v,u_{(*)}),\ldots ,\chi _m(v,u_{(*)})\in 
\mathbb{K}\big(H_j(u_{(i)},v): 1\leq i\leq m-1,\ 1\leq j\leq m\big).
\]
Without loss of generality we may assume that not all the $\chi _1,\ldots ,\chi _m$ are $0$.
Indeed there is a minor of $D$ of order $\nu \in (0,m)$ that is not zero and thus we can assume that $\nu =m-1$.
Now, fix $i\in \{1,\ldots ,m\}$ such that $\chi _i( v, u_{(*)})\neq 0$.
Then by Lemma \ref{pole}.$(2)$ there exists $ a:=( a_{(1)},\ldots , a_{(m-1)}) \in U_{\mathbb{K},(m-1)n}(\epsilon )$ 
such that
\[
\chi _1 ( v, a),\ldots ,\chi _m( v, a) \in M_{\mathbb{K},n} \text{ and } \chi _i ( v, a)\neq 0.
\]
We note that by hypothesis $f( a_{(1)}, v),\ldots ,f( a_{(m-1)}, v)\in \mathbb{K}(\Psi ( v))$,
therefore $\chi _1( v, a),\ldots ,\chi _m( v, a)\in \mathbb{K}(\Psi ( v))$. 
Since $\chi _i( v, a)\neq 0$, evaluating $ u_{(*)}$ at $ a$ we obtain an equation as in (\ref{E2}).
This concludes the proof.
\end{proof}

With the previous lemmas we can follow the proof of \cite[Chap. 5. \textsection 13. Theorem 1]{Siegel} and apply it to our 
context.

\begin{proposition}\label{S-BM-alg} Let $\Psi :=(\psi _0,\ldots ,\psi _n)\in M_{\mathbb{K},n}^{n+1}$ be convergent in 
$U_n(\epsilon)$.
Let $\psi _1,\ldots ,\psi _n$ be algebraically independent over $\mathbb{K}$ and let $\psi _0$ be algebraic over 
$\mathbb{K}(\psi _1,\ldots ,\psi _n)$. 
Let $f(u,v)\ \in M_{\mathbb{K},2n}$ be convergent in $U_{2n}(\epsilon )$ and algebraic over 
$\mathbb{K}(\Psi _{(u,v)})$.
If for each $k\in U_{\mathbb{K},n}(\epsilon )$ both $f(u,k)\in \mathbb{K}(\Psi (u))$ and $f(k,v)\in \mathbb{K}(\Psi (v))$ 
then $f(u,v)\in \mathbb{K}(\Psi _{(u,v)})$.
\end{proposition}
\begin{proof}
Let $\psi :=(\psi _1,\ldots ,\psi _n)$.
Let
\[
P(X)=X^h+P_1X^{h-1}+\ldots +P_h\in \mathbb{K}(\psi )[X].
\]
be the minimal polynomial of $\psi _0$ over $\mathbb{K}(\psi )$.
We note that $\mathbb{K}(\Psi )$ is isomorphic to $\mathbb{K}(X_1,\ldots ,X_n)[X]/(P(X))$.
For each $k\in U_{\mathbb{K},n}(\epsilon )$ let $f_{k}$ denote $f(u,k)$, then by hypothesis 
\[
f_{k} = S_{k,1}(\psi )\psi _0^{h-1}+S_{k,2}(\psi )\psi _0^{h-2}+\ldots +S_{k,h-1}(\psi )\psi _0
+S_{k,h}(\psi )
\]
for some $S_{ k,1},\ldots ,S_{ k,h}\in \mathbb{K}(X_1,\ldots ,X_n)$.
By Lemma \ref{S-BM} we only need to check that there exists $N\in \mathbb{N}$ such that for each 
$k\in U_{\mathbb{K},n}(\epsilon )$ each of $S_{k,1},\ldots ,S_{k,h}$ is the quotient of two polynomials in 
$\mathbb{K}[X_1,\ldots ,X_n]^{\leq N}$.

Fix $k\in U_{\mathbb{K},n}(\epsilon )$.
Let $\xi _1,\ldots ,\xi _h$ be the $h$ roots of $P(X)$.
For each $\alpha$ algebraic over $\mathbb{K}(\psi )$ let $\sigma (\alpha )$ denote its trace, hence 
$\sigma (\psi _0)=\xi _1+\ldots +\xi _h\in \mathbb{K}(\psi)$.
For each $i\in \{1,\ldots ,h\}$ we define
\[
f_{ k}^{(i)}:=S_{ k,1}(\psi )\xi _i^{h-1}+S_{ k,2}(\psi )\xi _i^{h-2}+\ldots 
+S_{ k,h-1}(\psi )\xi _i+S_{ k,h}(\psi ).
\]
Let
\[
L:=
\left[\begin{matrix}
\xi _1^{h-1} & \xi _2 ^{h-1} & \ldots & \xi _h ^{h-1}\\
\xi _1^{h-2} & \xi _2 ^{h-2} & \ldots & \xi _h ^{h-2}\\
\vdots & \vdots & \ddots & \vdots\\
\xi _1 & \xi _2 & \ldots & \xi _h\\
1 & 1 & \ldots & 1\\
\end{matrix}\right].
\]
An easy computation shows that $LL^t=\left[\begin{matrix} \sigma (\psi _0^{2h-i-j}) \end{matrix}\right]$
and so its coefficients belong to $\mathbb{K}(\psi )$.
Since $det(LL^t)=\prod _{1\leq i<j\leq h} (\xi _i-\xi _j)^2$ and $\mathbb{K}(\psi)$ is separable,
$LL^t$ is invertible.
We note that $\mathbb{K}(\psi )$ is isomorphic to $\mathbb{K}(X_1,\ldots ,X_n)$ because $\psi _1,\ldots ,\psi _n$ are algebraically 
independent over $\mathbb{K}$.
Hence in an abuse of notation we identify each $S_{k,i}$ with $S_{k,i}(\psi )$.
With this convention,
\[
\left[ \begin{matrix} f_{k}^{(1)},f_{k}^{(2)},\ldots ,f_{k}^{(h)}\end{matrix}\right]
=
\left[ \begin{matrix} S_{k,1},S_{k,2},\ldots ,S_{k,h}\end{matrix}\right] L
\]
and $\sigma (f_{k}\psi _0 ^j)=\sum _{i=1}^h f_{k}^{(i)}\xi _i ^j$ for each $j\in \mathbb{N}$, so
\[
\left[ \begin{matrix} S_{k,1},S_{k,2},\ldots ,S_{k,h}\end{matrix}\right]
=
\left[ \begin{matrix} \sigma (f_{k}\psi _0 ^{h-1}), \sigma (f_{k}\psi _0 ^{h-2}),\ldots ,\sigma (f_{k}\psi _0),
\sigma (f_{k})\end{matrix}\right] (LL^t)^{-1}.
\]
Since $L$ does not depend on $ k$, it is enough to show that there exists $N\in \mathbb{N}$ such that for each 
$ k\in U_{\mathbb{K},n}(\epsilon )$ each $\sigma (f_{ k}\psi _0 ^{h-1}),\ldots ,\sigma (f_{ k})$ can be written in the 
form $A(\psi )/B(\psi )$ for some $A,B \in \mathbb{K}[X_1,\ldots ,X_n]^{\leq N}$ and $B\neq 0$.

Now, we fix $j\in \{0,\ldots ,h-1\}$ and $ k\in U_{\mathbb{K},n}(\epsilon )$ and we check the statement above for 
$\sigma (f_{k}\psi _0 ^j)$.
Since $\psi _0$ is algebraic over $\mathbb{K}(\psi )$, by hypothesis both $f(u,v)$ and $\psi _0(u)$ are 
algebraic over $\mathbb{K}(\psi _{(u,v)})$.
Now we apply Lemma \ref{core} to $f(u,v)\psi _0(u) ^j$ to deduce that there exists $N\in \mathbb{N}$ such that 
for each $k\in U_{\mathbb{K},n}(\epsilon )$ the minimal polynomial of $f_{k}\psi _0 ^j$ over $\mathbb{K}(\psi )$ 
can be written in the form
\[
Y^h+\sum _{i=0}^{h-1} \frac{A_i(\psi )}{B_i(\psi )} Y^i
\]
for some $A_i,B_i \in \mathbb{K}[X_1,\ldots ,X_n]^{\leq N}$ and $B_i\neq 0$.
Finally, since $k$ was fixed and $\sigma (f_{ k}\psi _0 ^j)=-A_{h-1}(\psi )/B_{h-1}(\psi )$, we are done.
\end{proof}

\subsection{Some AAT results and the proof of Theorem \ref{T2}.}Next, we show additional properties for those elements of $M_{\mathbb{K},n}^n$ that admit an AAT.
We begin with a corollary of Lemma \ref{core}.
We shall use the notation introduced in Notation \ref{notation-AAT}.

\begin{corollary}\label{algebraic} Let $\phi \in M_{\mathbb{K},n}^n$ be convergent in $U_n(\epsilon )$.
 If $\phi$ admits an AAT then $\phi _{u+k}$ is algebraic over $\mathbb{K}(\phi  )$ for each 
 $ k\in U_{\mathbb{K},n}(\epsilon )$.
\end{corollary}
\begin{proof}
Let $f(u,v):=\phi (u+v)$.
Since $\phi$ admits an AAT and $\phi _{u+k}=f(u,k)\in M_{\mathbb{K},n}$ for all $ k\in U_n(\epsilon )$, we are under the hypothesis 
of Lemma \ref{core}.
\end{proof}

Although we will not use it, the proof of Lemma \ref{core} can be adapted to prove that if $\phi \in M_{\mathbb{K},n}^n$ admits 
an AAT then the formal derivative $\partial _{u_j} \phi _i$ is algebraic over $\mathbb{K}(\phi )$ for each $i,j\in \{1,\ldots ,n\}$.

\begin{lemma}\label{elimination}
Let $\phi ,\psi \in M_{\mathbb{K},n}^n$ and suppose that $\phi$ is algebraic over $\mathbb{K}(\psi )$.
If $\phi$ admits an AAT then $\psi$ admits an AAT.
The converse is also true provided $\phi_1,\ldots ,\phi _n$ are algebraically independent over $\mathbb{K}$.
\end{lemma}
\begin{proof}
Assume that $\phi$ admits an AAT, hence $\psi _1,\ldots ,\psi _n$ are algebraically independent over $\mathbb{K}$ because 
$\phi$ is algebraic over 
$\mathbb{K}(\psi )$.
To check that $\psi _{ u+ v}$ is algebraic over $\mathbb{K}(\psi _{(u,v)})$ it is enough to show that 
$\psi _{ u+ v}$ is algebraic over $\mathbb{K}(\phi _{ u+ v})$, $\phi _{ u+ v}$ is algebraic over 
$\mathbb{K}(\phi _{(u,v)})$ and $\phi _{(u,v)}$ is algebraic over $\mathbb{K}(\psi _{(u,v)})$.
The three conditions above are trivially satisfied because $\phi$ admits an AAT and both $\phi$ is algebraic over 
$\mathbb{K}(\psi)$ and $\psi$ is algebraic over $\mathbb{K}(\phi)$.

The converse follows by symmetry because if $\phi _1,\ldots ,\phi _n$ are algebraically independent over $\mathbb{K}$ then 
$\psi$ is algebraic over $\mathbb{K}(\phi )$.
\end{proof}

Now we adapt to our context a result of AAT due to H.A.Schwarz, see \cite[Chap. XXI. Art. 389]{Hancock} for details.

\begin{lemma}\label{Schwarz}
Let $\phi \in M_{\mathbb{K},n}^n$ be convergent in $U_n(\epsilon )$ and admitting an AAT.
Then there exist a finite subset $\mathcal{D}\subset U_{\mathbb{K},n}(\epsilon)$, 
$0\in \mathcal{D}$ and $\epsilon ' \in (0, \epsilon]$ such that each element of $\mathbb{K}(\phi _{u+d}:  d\in \mathcal{D})$ is convergent in $U_n(2\epsilon ')$,
and there exist $A_0,\ldots ,A_N\in \mathbb{K}(\phi _{(u+d,v+d)}:  d\in \mathcal{D})$
convergent in $U_{2n}(2\epsilon ')$ such that $\phi _{ u+ v}$ is algebraic over $\mathbb{K}(A_0,\ldots , A_N)$ 
and for each $\ell\in \{0,\ldots ,N\}$ 
\[
A_\ell ( u, v)=A_\ell ( u+ k, v- k)\text{ for all }  k\in U_{\mathbb{K},n}(\epsilon ').
\tag{$\dagger$}\label{dagger}
\]
\end{lemma}
\begin{proof}
Fix $i\in \{1,\ldots ,n\}$.
Let $\mathcal{S}_0:=\{ 0\}$ and $\mathbb{K}_0:=\mathbb{K}(\phi _{(u,v)} )$.
Let
\[
P_0(X)=X^{N_0+1}+\sum _{\ell=0}^{N_0} A_{0,\ell}( u, v) X^\ell
\]
be the minimal polynomial of $\phi _i( u+ v)$ over $\mathbb{K}_0$.
If each $A_{0,\ell}$ satisfies property (\ref{dagger}) for $\epsilon '=2^{-1}\epsilon$ then we are done for this $i$ letting 
$\epsilon ':=2^{-1}\epsilon$, $\mathcal{D}:=\mathcal{S}_0$ and $A_\ell:=A_{0,\ell}$ for each $0\leq \ell\leq N_0$.
Otherwise, there exists $k_1\in U_{\mathbb{K},n}(2^{-1}\epsilon )$ such that
\[
Q_0(X):= X^{N_0+1}+\sum _{\ell=0}^{N_0} A_{0,\ell}( u, v)X^\ell - 
X^{N_0+1}-\sum _{\ell=0}^{N_0} A_{0,\ell}(u+k_1,v-k_1)X^\ell 
\]
is not zero.
Since $u+v=(u+k_1)+(v-k_1)$, we deduce that $\phi _i( u+ v)$ is a root of $Q_0(X)$.
Let $\mathcal{S}_1:=\mathcal{S}_0\cup \{k_1, -k_1\}$ and 
$\mathbb{K}_1:=\mathbb{K}(\phi _{u+k,v+k}:  k\in \mathcal{S}_1)$.
By definition $\mathbb{K}_0\subset \mathbb{K}_1$.
Let 
\[
P_1(X)=X^{N_1+1}+\sum _{\ell=0}^{N_1} A_{1,\ell}(u,v) X^\ell
\]
be the minimal polynomial of $\phi _i( u+ v)$ over $\mathbb{K}_1$.
We note that the elements of $\mathbb{K}_1$ are convergent in $U_{2n}(2^{-1}\epsilon )$.
If each $A_{1,\ell}$ satisfies property (\ref{dagger}) for $\epsilon '=2^{-2}\epsilon$ then we are done for this $i$ letting 
$\epsilon ':=2^{-2}\epsilon$, $\mathcal{D}:=\mathcal{S}_1$ and $A_\ell:=A_{1,\ell}$ for each $0\leq \ell\leq N_1$.
Otherwise, we can repeat the process to obtain sets $\mathcal{S}_2$, $\mathcal{S}_3$ and so on where the set 
$\mathcal{S}_r$ is obtained from the set $\mathcal{S}_{r-1}$ as
\[
\mathcal{S}_r:=\mathcal{S}_{r-1}\cup \{ k + k_r: k\in \mathcal{S}_{r-1}\}\cup 
\{ k - k _r: k\in \mathcal{S}_{r-1}\}
\]
for some $k_r\in U_{\mathbb{K},n}(2 ^{-r}\epsilon )$ such that $Q_{r-1}$ is not $0$. Similarly, we obtain $\mathbb{K}_r:=\mathbb{K}(\phi_{u+k,v+k}:k\in \mathcal{S}_r)$ whose elements are convergent in $U_{2n}(2^{-r}\epsilon)$.
Since in the $r$ repetition the degree of $P_r$ is smaller than that of $P_{r-1}$, this process eventually stops, say at step $s$.
Letting $\epsilon ':=2^{-s-1}\epsilon$, $\mathcal{D}:=\mathcal{S}_s$ and $A_\ell:=A_{s,\ell}$ for each $0\leq \ell\leq N_s$, 
we are done for this $i$.
The elements $A_0,\ldots ,A_{N_s}$ are convergent in $U_{2n}(2\epsilon ')$ since they are elements of $\mathbb{K}_s$.

For each $i$ $(1\leq i \leq n)$ denote by $\epsilon '_i$, $\mathcal{D}_i$ and $A_{0}^{i},\ldots ,A_{N_i}^{i}$ the elements 
$\epsilon '$, $\mathcal{D}$ and $A_{1},\ldots ,A_N$ previously obtained for that choice of $i$.
To complete the proof, take $\mathcal{D}:=\cup _i \mathcal{D}_i$, $\epsilon ' :=\min _i \{ \epsilon '_i\}$, and let 
$\{A_0,\ldots ,A_N\}$ be the union of the sets $\{A_0^{i},\ldots ,A_{N_i}^{i}\}$.
\end{proof}

We need one more lemma before proving Theorem \ref{T2}.

\begin{lemma}\label{K(Psi)} Let $\phi \in M_{\mathbb{K},n}^n$ be convergent in $U_n(\epsilon )$ admitting an AAT.
Then there exist $\epsilon ''\in (0,\epsilon ]$ and $\Psi :=(\psi _0,\ldots ,\psi _n)\in M_{\mathbb{K},n}^{n+1}$ 
convergent in $U_n(\epsilon '')$ and algebraic over 
$\mathbb{K}(\phi )$ such that $\psi :=(\psi _1,\ldots ,\psi _n)$ admits an AAT, 
$\psi _0$ is algebraic over $\mathbb{K}(\psi )$ and for each $f \in \mathbb{K}(\Psi )$ there exists 
$\delta \in (0, \epsilon '']$ such that for each $k\in U_{\mathbb{K},n}(\delta )$, $f_{u+k}\in \mathbb{K}(\Psi)$ and 
$f_{u+k}$ is convergent in $U_n(\epsilon '')$.
\end{lemma}
\begin{proof}
We will define a field $\mathbb{L}$ and we will check that this $\mathbb{L}$ satisfies the 
conditions of the theorem. 
Once this is done, we will find $\Psi$ such that $\mathbb{L}=\mathbb{K}(\Psi )$.

Let $\epsilon ' \in (0, \epsilon]$, $\mathcal{D}\subset U_{\mathbb{K},n}(\epsilon )$ and 
$A_0,\ldots ,A_N \in \mathbb{K}(\phi _{(u+d,v+d)}:  d\in \mathcal{D})$ 
be the ones obtained applying Lemma \ref{Schwarz} to $\phi$.
By Lemma \ref{pole}.$(1)$ there exists an open dense subset $U\subset U_{\mathbb{K},n}(\epsilon ')$ such that
\[
U\subset \{  a\in U_{\mathbb{K},n}(\epsilon '): \phi ( d+ a)\in \mathbb{K}^n \text{ for all }  d\in \mathcal{D}\}
\]
and
\[
U\subset \{ a\in U_{\mathbb{K},n}(\epsilon '):  A_0 ( u,  a),\ldots ,A_N ( u,  a) \in M_{\mathbb{K},n} \}.
\]
In particular $U\subset \{  a \in U_{\mathbb{K},n}(\epsilon '): \phi ( a)\in \mathbb{K}^n\}$ since $0\in \mathcal{D}$.
Since $U$ is open there exist $b\in U$ and $\epsilon '' \in (0,\epsilon '- \|  b\|]$ such that 
\[
V:=\{  a\in U_{\mathbb{K},n}(\epsilon ') : \|  a - b\| < \epsilon '' \}\subset U.
\]
Fix such $b$.
Then, for each $a \in U_{\mathbb{K},n}(\epsilon '' )$, each $A_\ell(u,b+ a)$ is an element of $M_{\mathbb{K},n}$.
We note that since each $A_\ell(u,v)$ is convergent in $U_{2n}(2\epsilon ')$ and by definition of $b$ and $\epsilon ''$, 
each $A_\ell(u,b+a)$ is convergent in $U_n(\epsilon ')$ for each $a\in U_{\mathbb{K},n}(\epsilon '' )$.
Also, since each $A_\ell$ satisfies the property (\ref{dagger}) of Lemma \ref{Schwarz},
\[
A_\ell(u,b+ a)=A_\ell( u+ a, b)  
\text{ for all }  a\in U_{\mathbb{K},n}(\epsilon '' ).  \tag{$\dagger_{ b}$}\label{daggerb}
\]
For each $\ell\in \{0,\ldots ,N\}$ we define $B_\ell( u):=A_\ell( u, b)$.
Let
\[
\mathbb{L} :=\mathbb{K}((B_\ell)_{u+a}:  a\in U_{\mathbb{K},n}(\epsilon '' ), 0\leq \ell \leq N).
\]
Since for each $a\in U_{\mathbb{K},n}(\epsilon '' )$ each $A_\ell(u,b+a)$ is convergent in $U_n(\epsilon ')$, 
by property (\ref{daggerb}) all the elements of $\mathbb{L}$ are convergent in $U_n(\epsilon ')$ and 
in particular in $U_n(\epsilon '')$.

We are going to show that 
\[
\mathbb{L}\subset \mathbb{K}(\phi _{u+d}:  d\in \mathcal{D}) 
\]
and that each element of $\mathbb{L}$ 
is algebraic over $\mathbb{K}(\phi )$.
Fix $\ell\in \{0,\ldots ,N\}$ and $a\in U_{\mathbb{K},n}(\epsilon '' )$.
We recall from Lemma \ref{Schwarz} that $A_\ell(u,v)$ is convergent in $U_{2n}(2\epsilon ')$ and 
$A(u,v)\in \mathbb{K}(\phi  _{(u+d,v+d)} :  d\in \mathcal{D})$.
Hence we can evaluate $A_\ell(u,v)$ at $v=b+a$ to deduce that $A_\ell(u,b+a)\in \mathbb{K}(\phi _{u+d}:  d\in \mathcal{D})$.
Thus, by property (\ref{daggerb}), $A_\ell(u+a,b)\in \mathbb{K}(\phi _{u+d}:  d\in \mathcal{D})$.
Hence, $\mathbb{L}\subset \mathbb{K}(\phi _{u+d}:  d\in \mathcal{D})$ and therefore, by Corollary \ref{algebraic}, 
each element of $\mathbb{L}$ is algebraic over $\mathbb{K}(\phi )$.

Next, we show that $\phi _1( u+ b),\ldots ,\phi _n( u+ b)$ are algebraically independent over $\mathbb{K}$.
Let $P\in \mathbb{K}[X_1,\ldots ,X_n]$ such that $P(\phi _{u+b})=0$.
By notation $P(\phi _{u+b}(a))=0$ if and only if $P(\phi (a+b))=0$, for 
$ a\in U_{\mathbb{K},n}(\epsilon '')$.
Hence
\[
V \subset \{  a\in U_{\mathbb{K},n}(\epsilon ) : P(\phi ( a))\in \mathbb{K} \text{ and } P(\phi ( a))=0\}.
\]
Since $V$ is open in $U_{\mathbb{K},n}(\epsilon )$, $P(\phi )=0$ by the identity principle.
Since $\phi_1$,\ldots ,$\phi _n$ are algebraically independent over $\mathbb{K}$, $P=0$ and we are done.

Next, we show that $\mathbb{L}$ is finitely generated over $\mathbb{K}$ and its transcendence degree is $n$.
Firstly, we note that $\phi$ is algebraic over $\mathbb{K}(\phi _{u+b})$ because 
the coordinate functions of $\phi _{u+b}$ are algebraically independent over 
$\mathbb{K}$ and $\phi _{u+b}$ is algebraic over $\mathbb{K}(\phi )$ by Corollary \ref{algebraic}.
Since $\phi _{ u+ v}$ is algebraic over $\mathbb{K}(A_0,\ldots ,A_N)$, evaluating each $A_\ell(u,v)$ at $v=b$ we deduce that 
$\phi _{u+b}$ is algebraic over $\mathbb{K}(B_0,\ldots ,B_N)$.
Therefore, $\phi$ is algebraic over $\mathbb{K}( B_0,\ldots ,B_N)$.
On the other hand, $\mathbb{K}(B_0,\ldots ,B_N)$ is a subset of $\mathbb{K}(\phi _{u+d}:  d\in \mathcal{D})$ and the 
latter field is algebraic over $\mathbb{K}(\phi )$ by Corollary \ref{algebraic}.
Hence the three fields have transcendence degree $n$ over $\mathbb{K}$.
Now, $\mathcal{D}$ is finite and 
\[
\mathbb{K}(B_0,\ldots ,B_N)\subset \mathbb{L} \subset \mathbb{K}(\phi _{u+d}:  d\in \mathcal{D}),
\]
therefore, $\mathbb{L}$ is finitely generated over $\mathbb{K}$ and its transcendence degree is $n$.

Let $f \in \mathbb{L}$, we now check that there exists $\delta >0$ such that such that for every  
$a\in U_{\mathbb{K},n}(\delta )$, $f_{u+a}\in \mathbb{L}$ and $f_{u+a}$ is convergent in $U_n(\epsilon '')$.
Since $f \in \mathbb{L}$, there exist $m\in \mathbb{N}$, $\ell(1),\ldots ,\ell(m)\in \{0,\ldots ,N\}$ and 
$a_1,\ldots , a_m\in U_{\mathbb{K},n}(\epsilon '' )$ such that $f$ is a rational function of 
$(B_{\ell(1)})_{u+a_1}$, \ldots ,$(B_{\ell(m)})_{u+a_m}$.
Take $\delta >0$ such that $\delta <\epsilon '' -\max\{ \| a_1\|,\ldots ,\|  a_m\|\}$.
Then, for all $ a\in U_{\mathbb{K},n}(\delta )$, $f _{u+a}\in \mathbb{L}$ and $f_{u+a}$ is convergent in $U_n(\epsilon '')$.

Finally, by the primitive element theorem there exist $\psi _1,\ldots ,\psi _n\in \mathbb{L}$ algebraically independent 
over $\mathbb{K}$ and $\psi _0$ algebraic over $\mathbb{K}(\psi _1,\ldots ,\psi _n)$ such that 
$\mathbb{L}=\mathbb{K}(\psi _0,\psi _1,\ldots ,\psi _n)$.
Now, since all the elements of $\mathbb{L}$ are algebraic over $\mathbb{K}(\phi )$, $\psi:= (\psi _1,\ldots ,\psi _n)$ admits 
an AAT by Lemma \ref{elimination}.
\end{proof}

We now have all the ingredients to prove the main result of this section.

\begin{theorem}\label{T2}
Let $\phi \in M_{\mathbb{K},n}^n$ admitting an AAT.
Then there exists $\psi :=(\psi _1,\ldots ,\psi _n)\in M_{\mathbb{K},n}^n$ admitting an AAT and algebraic over $\mathbb{K}(\phi)$ 
and $\psi _0\in M_{\mathbb{K},n}$ algebraic over $\mathbb{K}(\psi )$ such that
\begin{enumerate}
\item[$(1)$] for each $f\in \mathbb{K}(\psi _0,\ldots ,\psi _n)$ there exists $R\in \mathbb{K}(X_1,\ldots ,X_{2(n+1)})$ such that
\[
f ( u+ v)=R \big(\psi _0( u),\ldots ,\psi _n( u),\psi _0( v),\ldots ,\psi _n( v)\big),
\]
\item[$(2)$] and each $\psi _0,\ldots ,\psi _n$ is the quotient of two power series, both convergent in all $\mathbb{C}^n$.
\end{enumerate}
\end{theorem}
\begin{proof}
Let $\phi:=(\phi _1,\ldots ,\phi _n)\in M_{\mathbb{K},n}^n$ admitting an AAT.
Take $\epsilon$ such that $\phi$ is convergent in $U_n(\epsilon )$.
Applying Lemma \ref{K(Psi)} we obtain $\epsilon ''\in (0,\epsilon ]$ and 
$\Psi :=(\psi _0,\ldots ,\psi _n)\in M_{\mathbb{K},n}^{n+1}$ as in the lemma.
We next check that this $\Psi$ satisfies the conditions of the theorem.

$(1)$ Fix a non constant $f\in \mathbb{K}(\Psi )$.
Fix $\delta \in (0, \epsilon '']$ such that $f_{u+k}\in \mathbb{K}(\Psi )$ for each $ k\in U_n(\delta)$ as in Lemma \ref{K(Psi)}.
Let $\varepsilon < \delta$ and such that $f_{ u+ v}$ is convergent in $U_{2n}(\varepsilon )$.
It is enough to show that $f_{u+v}\in M_{\mathbb{K},2n}$ is algebraic over $\mathbb{K}(\Psi _{(u,v)})$ since then we can apply 
Proposition \ref{S-BM-alg} noting that both $f_{u+k}\in \mathbb{K}(\Psi (u))$ and $f_{v+k}\in \mathbb{K}(\Psi (v))$ for each 
$k\in U_{\mathbb{K},n}(\varepsilon )$.
With this aim, take $g_2,\ldots ,g_n\in \mathbb{K}(\psi )$ such that $f,g_2,\ldots ,g_n$ are algebraically independent over $\mathbb{K}$.
Let $g:=(f,g_2,\ldots ,g_n)$ and we note that $g$ is algebraic over $\mathbb{K}(\psi )$.
Since $\psi$ admits an AAT, $g$ admits an AAT by Lemma \ref{elimination}.
Hence $g_{u+v}$ is algebraic over $\mathbb{K}(g _{(u,v)})$ and therefore over $\mathbb{K}(\Psi _{(u,v)})$.
This concludes the proof of $(1)$.

\smallskip

\noindent $(2)$ We may assume that $\psi _0\neq 0$.
Fix $i\in \{0,\ldots ,n\}$.
We have already shown that $\psi _i( u+ v)\in \mathbb{K}(\Psi _{(u,v)})$.
Let $A( u, v):=\psi _i( u+ v)$.
By Lemma \ref{K(Psi)} and by reducing $\epsilon $ if necessary, we may assume that $\Psi$ is convergent in 
$U_n(\epsilon )$ and $\mathbb{K}(\Psi _{u+k})\subset \mathbb{K}(\Psi)$ for all $k\in U_{\mathbb{K},n}(\epsilon)$.
We show that there exists $ p\in U_{\mathbb{K},n}(\epsilon)$ such that 
\[
A(u+p,u-p)\in M_{\mathbb{K},n}.
\]
Take $\alpha ,\beta \in A_{\mathbb{K},2n}$, $\beta \neq 0$ such that 
$A(u,v)=\frac{\alpha (u,v) }{\beta (u,v)}$.
Suppose for a contradiction that $\beta (u+k,u-k)=0$ for all $k\in U_{\mathbb{K},n}(\epsilon)$.
Then 
\[
\beta \left(\frac{a+b}{2}+\frac{a-b}{2},\frac{a+b}{2}-\frac{a-b}{2}\right) =0
\]
for all $a,b\in U_{\mathbb{K},n}(\epsilon/2)$.
So $\beta (a,b)=0$ for all $( a, b)\in U_{\mathbb{K},n}(\epsilon/2)$ and hence $\beta =0$, a contradiction.
Then
\[
\psi _i(2 u)=A( u+ p, u- p)\in \mathbb{K}(\Psi _{u+p}( u),
\Psi _{u-p}( u))\subset \mathbb{K}(\Psi (u)).
\]

By induction we deduce that 
\[
\psi _0( u),\ldots ,\psi _n( u)\in \mathbb{K}(\Psi (2^{-N} u))
\]
for each $N\in \mathbb{N}$.
Hence since $\Psi (2^{-N} u)$ is convergent in $U_n(2^N\epsilon)$, $\Psi$ is also convergent in $U_n(2^N\epsilon)$. Thus each $\psi_i$ is a meromorphic function and therefore by Remark \ref{rmkPoincare} it is the quotient of two power series convergent in all $\mathbb{C}^n$.
\end{proof}

We end this section with some basic properties of differentials that we will need for the proof of Theorem \ref{T1}.	
We introduce the following notation.
Let $u:=(u_1,\ldots ,u_n)$ be $n$ variables, then for any $j\in \{1,\ldots ,n\}$ we denote 
$\partial _{u_j}:A_{\mathbb{K},n}\rightarrow A_{\mathbb{K},n}$ the formal derivative in 
the variable $u_j$.
As $\partial _{u_j}$ is a derivation of $A_{\mathbb{K},n}$ it induces a derivation on $M_{\mathbb{K},n}$.
Given $\phi \in M_{\mathbb{K},n}$ let $d\phi$ be the differential of $\phi$, {\it i.e.} 
$[\partial _{u_1}\phi,\ldots ,\partial _{u_n}\phi ]$.
We note that if $^a \phi$ is the germ of an analytic function at $ 0$ then $^a (d\phi)$ is $\nabla \, {^a\phi}$, 
the gradient of $^a \phi$.

\begin{lemma}\label{dindependence}
Let $\phi \in M_{\mathbb{K},n}^m$ such that $d\phi _1,\ldots ,d\phi _m$ are linearly independent over $M_{\mathbb{K},n}$.
Then 
\begin{enumerate}
\item[$(1)$] $\phi _1,\ldots ,\phi _m$ are algebraically independent over $\mathbb{K}$.
\item[$(2)$] If $\psi \in M_{\mathbb{K},n}^m$ and $\phi$ is algebraic over $\mathbb{K}(\psi )$ then 
$d\psi _1,\ldots ,d\psi _m$ are linearly independent over $M_{\mathbb{K},n}$.
\end{enumerate}
\end{lemma}
\begin{proof}
$(1)$ Suppose that $\phi _1,\ldots ,\phi _m$ are algebraically dependent over $\mathbb{K}$, then we may assume that $\phi _m$ 
is algebraic over $\mathbb{K}(\phi _1,\ldots ,\phi _{m-1})$.
If $\phi _m$ is constant then $d\phi _m=0$ and the lemma is proved, so we may assume that $\phi _m\notin \mathbb{K}$.

Let $P$ be the minimal polynomial of $\phi _m$ over $\mathbb{K}(\phi_1,\ldots ,\phi _{m-1})$.
We note that $P(\phi _m)$ and $\frac{\partial P}{\partial X} (\phi _m)$ are elements of $M_{\mathbb{K},n}$ and therefore
\[
dP(\phi _m)=\sum _{i=1}^{m-1} g_i\, d\phi _i+ \frac{\partial P}{\partial X} (\phi _m)\, d\phi _m
\]
for some $g_1,\ldots ,g_{m-1}\in M_{\mathbb{K},n}$.
Since $P$ is the minimal polynomial of $\phi _m$, $P(\phi _m)=0$.
This implies that $dP(\phi _m)$ is the vector $[0,\ldots ,0]$ of $M_{\mathbb{K},n}^n$ and there exist
$h_1,\ldots ,h_{m-1}\in M_{\mathbb{K},n}$ such that
\[
d\phi _m= \sum _{i=1}^{m-1} h_i\, d\phi _i.
\]

\smallskip

\noindent $(2)$ For every $i \in \{1,\ldots ,m\}$ we have that $\phi _i$ is not constant because $d\phi _i\neq 0$.
Since $\phi$ is algebraic over $\mathbb{K}(\psi )$, by the proof of $(1)$ for each $i \in \{1,\ldots ,m\}$ 
there exist $g_{i,1},\ldots ,g_{i,m}\in M_{\mathbb{K},n}$ such that
\[
d\phi _i=\sum _{j=1}^m g_{i,j}\, d\psi _j.
\]
Therefore there exists a $m\times m$ matrix $G$ with coefficients in $M_{\mathbb{K},n}$ such that
\[
\left[ \begin{matrix} d\phi _1\\ \vdots \\ d\phi _m\end{matrix}\right] = 
G\left[ \begin{matrix} d\psi _1 \\ \vdots \\ d\psi _m\end{matrix}\right].
\]
Since $\phi _1,\ldots ,\phi _m$ are algebraically independent over $\mathbb{K}$ by $(1)$,
we have that $\psi$ is algebraic over $\mathbb{K}(\phi )$ and hence, by symmetry, there exists a $m\times m$ matrix $H$ 
with coefficients in $M_{\mathbb{K},n}$ such that
\[
\left[ \begin{matrix} d\psi _1\\ \vdots \\ d\psi _m\end{matrix}\right] = 
H\left[ \begin{matrix} d\phi _1 \\ \vdots \\ d\phi _m\end{matrix}\right].
\]
Hence $HG=GH=Id$, and so $d\psi _1,\ldots ,d\psi _m$ are linearly independent over $M_{\mathbb{K},n}$.
\end{proof} 

\section{Periods of real meromorphic maps.}\label{meromorphic maps}

This section has two different purposes. Firstly, after recalling basic definitions and properties of meromorphic and real meromorphic functions, we give functorial versions of the results in Section \ref{AAT} and we prove Theorem \ref{T1}. Secondly, we will introduce some definitions and prove some technicals lemmas related to periods of meromorphic maps  from $\mathbb{C}^n$ to $\mathbb{C}^n$ that will be relevant to describe Nash atlas for $(\mathbb{R}^n,+)$ in the next sections.

\medskip

\subsection{Locally Nash groups and AAT}We begin recalling some concepts of analytic and meromorphic functions of several variables. We use the definitions and notations introduced at the beginning of Section 3 and recall that the elements of $\mathcal{M}_{\mathbb{C},n}(\mathbb{C}^n)$ are the meromorphic functions.  Let $U\subset \mathbb{C}^n$ be an open connected neighborhood of $0$. We say that an analytic function $f:U\rightarrow \mathbb{C}$ is a 
\emph{real analytic function} if $f(\mathbb{R}^n\cap U)\subset \mathbb{R}$.
A meromorphic function $f:\mathbb{C}^n\rightarrow \mathbb{C}$ is \emph{a real meromorphic function} if there exist real analytic functions $g,h:\mathbb{C}^n\rightarrow \mathbb{C}$, 
with $h$ not identically zero, such that $f=g/h$. Real analytic and real meromorphic maps are defined in the obvious way. 

Analytic functions can be characterized in terms of real analytic functions since for any analytic function $f:U\rightarrow \mathbb{C}$ there exist real analytic functions $\text{Re}(f),\text{Im}(f):U\rightarrow \mathbb{C}$ such that $f=\text{Re}(f)+i\text{Im}(f)$, and similarly for meromorphic functions.
We also remind that an analytic map $f:U\rightarrow \mathbb{C}^m$ is a real analytic function if and only 
${^t}f\in A_{\mathbb{R},n}^m$, and similarly for real meromorphic functions (with ${^t}f\in M_{\mathbb{R},n}^m$).

\medskip

Let $U\subset \mathbb{K}^n$ be an open connected neighborhood of $0$. Note that  $f_1,\ldots ,f_n\in \mathcal{M}_{\mathbb{K},n}(U)$ are algebraically independent over $\mathbb{K}$ if and only if ${^t f_1},\ldots ,{^t f_m}$ 
are algebraically independent over $\mathbb{K}$. Given $g=(g_1,\ldots,g_m)\in (\mathcal{M}_{\mathbb{K},n}(U))^m$, we say that $f=(f_1,\ldots,f_\ell)\in (\mathcal{M}_{\mathbb{K},n}(U))^\ell$ is \emph{algebraic} over $\mathbb{K}(g):=\mathbb{K}(g_1,\ldots,g_m)$ if each $f_i$ is algebraic over $\mathbb{K}(g)$, and again this is true if and only if ${^ t}f$ is algebraic over $\mathbb{K}({^t g})$.

We will say $f\in (\mathcal{M}_{\mathbb{K},n}(U))^n$ admits an \emph{algebraic addition theorem} 
if ${^t f}\in M_{\mathbb{K},n}^n$ admits an AAT. Not every element 
of $M_{\mathbb{C},n}^n$ admitting an AAT comes from a meromorphic map $f:\mathbb{C}^n\rightarrow \mathbb{C}^n$ admiting an AAT.
An example of this is the function $u\mapsto \sqrt{u+1}$, that although is not a meromorphic function its Taylor power series 
expansion at $0$ admits an AAT. 

\medskip

Now, we rewrite Corollary \ref{algebraic} and Lemma \ref{elimination} in terms of meromorphic functions.

\begin{corollary}\label{algebraic for functions}
Let $f,g: \mathbb{C}^n\rightarrow \mathbb{C}^n$ be meromorphic maps such that $f$ is algebraic over $\mathbb{C}(g)$.
\begin{enumerate}
\item[$(1)$] If $f$ admits an AAT then $f(u+a)$ is algebraic over $\mathbb{C}(f(u))$ for each $a\in \mathbb{C}^n$.
\item[$(2)$] If $f$ admits an AAT then $g$ admits an AAT.
In particular $f(u+a)$ admits an AAT for each $a\in \mathbb{C}^n$.
\end{enumerate}
\end{corollary}

We recall that the only analytic structure on $(\mathbb{R}^n,+)$ is the standard one (the one given by the identity map) and that 
its compatible charts are given exactly by the analytic diffeomorphisms.
In what follows we will use these facts without further mention.
Next, we relate AAT to properties of analytic groups, as mentioned before the proof of Fact \ref{compatibility0}.

\begin{lemma}\label{AAT-star}
Let $(U,\phi )$ be a chart of the identity of $(\mathbb{R}^n,+)$ compatible with its standard analytic structure.
Then the following are equivalent:
\begin{enumerate}
\item[$(1)$] there exists an open neighborhood of the identity $U'\subset U$ such that
\[
\phi \circ + \circ (\phi ^{-1},\phi ^{-1}):\phi (U')\times \phi (U')\rightarrow \phi (U) :
(x,y)\mapsto \phi (\phi ^{-1}(x)+\phi ^{-1}(y))
\]
is a Nash map, and therefore by Fact  \ref{compatibility0} there exists an open neighborhood $V\subset U$ of $0$ such that $(\mathbb{R}^n,+,\phi|_V)$ is a locally Nash group. 
\item[$(2)$] $\phi \in \mathcal{O}_{\mathbb{R},n}(U)$ admits an AAT.
\end{enumerate}
\end{lemma}
\begin{proof}
$(1)$ implies $(2)$:
By hypothesis $\phi (U')$ is semialgebraic, since it is the projection of the domain of a semialgebraic map. 
Fix $i\in \{1,\ldots ,n\}$.
As we have mentioned in the definition of Nash map, this hypothesis implies that there exists 
$P_i\in \mathbb{R}[X_1,\ldots ,X_{2n+1}]$, $P_i\neq 0$, such that
\[
P_i(x_1,\ldots ,x_n,y_1,\ldots ,y_n,\phi _i(\phi ^{-1}(x) + \phi ^{-1}(y)))\equiv 0 \text{ on } \phi (U')\times \phi (U')
\]
where $x:=(x_1,\ldots ,x_n)$ and $y:=(y_1,\ldots ,y_n)$.
Since $\phi$ is a diffeomorphism, letting $u:=\phi ^{-1} (x)$ and $v:=\phi ^{-1}(y)$ we deduce that 
\[
P_i(\phi _1(u),\ldots ,\phi _n(u),\phi _1(v),\ldots ,\phi _n(v),\phi _i(u+v))\equiv 0 \text{ on } U'\times U'.
\]
Since $\phi$ is a diffeomorphism, the coordinate functions $\phi _1,\ldots ,\phi _n$ are clearly algebraically independent.
So $\phi$ admits an AAT.

\smallskip

$(2)$ implies $(1)$:
Fix $i\in \{1,\ldots ,n\}$.
If $\phi$ admits an AAT then there exists $P_i\in \mathbb{R}[X_1,\ldots ,X_{2n+1}]$, $P_i\neq 0$, such that
\[
P_i(\phi _1(u),\ldots ,\phi _n(u),\phi _1(v),\ldots ,\phi _n(v),\phi _i(u+v))\equiv 0 \text{ on } U'\times U'
\]
for some open neighborhood of the identity $U'\subset U$.
Since $\phi$ is a diffeomorphism we can let $x:=\phi (u)$ and $y:=\phi (v)$ and argue as before once we shrink $U$ to make it 
semialgebraic.
\end{proof}

We can now justify the notation of $(\mathbb{R}^n,+,f)$ given in the introduction for a locally Nash group structure on $(\mathbb{R}^n,+)$. Indeed, if $f:\mathbb{C}^n\rightarrow \mathbb{C}^n$ is a real meromorphic map such that
\begin{itemize}
 \item[1)] $f$ is real meromorphic and admits an AAT, and
 \item[2)] there exist $k\in \mathbb{R}^n$ and an open neighborhood $U\subset \mathbb{R}^n$ of $0$ such that
\[
\psi : U\rightarrow \mathbb{R}^n: u\mapsto \psi(u):=f(u+k) 
\]
is an analytic diffeomorphism,
\end{itemize}
then by Lemma \ref{AAT-star} there exists an open neighborhood $V\subset U$ of $0$ such that  $(\mathbb{R}^n,+,\psi|_V)$ is a 
locally Nash group. Note that $f$ satisfies  1) and 2) here if and only if it satisfies 1) and 2) in the introduction.

It remains to check that the locally Nash group structure is independent of  $k$ and the domains $U$ and $V$, that is, we 
have to show that given a real meromorphic map $f:\mathbb{C}^n\rightarrow \mathbb{C}^n$ admitting an AAT and given $k_1,k_2\in \mathbb{R}^n$ such that
\[
\psi_1: U_1\rightarrow \mathbb{R}^n: u\mapsto f(u+k_1), \quad \psi_2: U_2\rightarrow \mathbb{R}^n: u\mapsto f(u+k_2),
\]
satisfy conditions 2) above, we have that $(\mathbb{R}^n,+,\psi_1|_{V_1})$ and $(\mathbb{R}^n,+,\psi_2|_{V_2})$ are isomorphic as 
locally Nash groups (where $V_1\subset U_1$ and $V_2\subset U_2$ are given by Lemma \ref{AAT-star}).
By Lemma \ref{algebraic for functions}.$(1)$, $\psi_1$ is algebraic over $\mathbb{C}(\psi_2)$.
Since both $\psi_1$ and $\psi_2$ are real analytic maps, $\psi_1$ is algebraic over $\mathbb{R}(\psi_2)$ on some 
neighborhood of $0$ and hence by Proposition \ref{compatibilityAut} the identity map is a locally Nash isomorphism between
$(\mathbb{R}^n,+,\psi_1|_{V_1})$ and $(\mathbb{R}^n,+,\psi_2|_{V_2})$.

\medskip

\emph{ Henceforth when we write $(\mathbb{R}^n,+f)$ where $f:\mathbb{C}^n\rightarrow \mathbb{C}^n$ is a real meromorphic function that admits an AAT, we are also assuming that $f$ satisfies property 2) above.}

\medskip

This convention is useful since now we can denote by $(\mathbb{R},+,\wp _{<1,i>}(x))$ the locally Nash
group $(\mathbb{R},+,\wp _{<1,i>}(x+a)|_U)$ where $U$ is a sufficiently small neighborhood of the identity
and $a\in \mathbb{R}$ is also sufficiently small.
We note that without this convention the notation $(\mathbb{R},+,\wp _{<1,i>}(x))$ would not make sense, since the map 
$\wp _{<1,i>}(x)$ is not even a local diffeomorphism at $0$.

\medskip

Now we are ready to prove one of the main results of the paper.

\begin{theorem}\label{T1}
Every simply connected $n$-dimensional abelian locally Nash group is locally Nash isomorphic to some
$(\mathbb{R}^n,+, f )$, where $f:\mathbb{C}^n\rightarrow \mathbb{C}^n$ is a real meromorphic map admitting an AAT.
\end{theorem}
\begin{proof}
Let $(G,\cdot)$ be a simply connected $n$-dimensional abelian locally Nash group equipped with a Nash atlas 
$\mathcal{B}:=\{(W_i,\phi _i)\}_{i\in I}$. In particular $G$ is an analytic group with this atlas and therefore there exists an isomorphism of analytic groups
\[
\rho :(G,\cdot) \rightarrow (\mathbb{R}^n ,+),
\]
where $(\mathbb{R}^n,+)$ is equipped with its unique analytic group structure, the standard one (see e.g. \cite[2.19]{Adams}).
Since $\mathcal{B}$ is a Nash atlas for $(G,\cdot )$ we have that
\[
\mathcal{A}:=\{ \rho(W_i), \phi _i\circ \rho^{-1}\}_{i\in I}
\]
is a Nash atlas for $(\mathbb{R} ^n , +)$ compatible with its standard analytic structure.
Moreover, $(G,\cdot)$ equipped with $\mathcal{B}$ is clearly locally Nash 
isomorphic to $(\mathbb{R}^n,+)$ equipped with $\mathcal{A}$.

Now, consider a chart of the identity $(U,\phi)\in \mathcal{A}$.
Firstly, note that as analytic chart $(U,\phi)$ must be compatible with the standard analytic structure of 
$(\mathbb{R}^n,+)$ and hence $\phi$ is an analytic diffeomorphism.
Also, being a chart of a locally Nash group structure, it satisfies condition $(1)$ of Lemma \ref{AAT-star} and hence $\phi$ admits an AAT.
Now, we apply Theorem \ref{T2} to ${^t} \phi$, the power series expansion of $\phi$ at $0$, to obtain 
$\psi :=(\psi _1,\ldots ,\psi _n)\in M_{\mathbb{R},n}^n$ convergent in $\mathbb{C}^n$ and admitting an AAT 
such that ${^t} \phi$ is algebraic over $\mathbb{R}(\psi)$. Note that since $\phi$ is an analytic diffeomorphism,
\[
{\phi} _*:T_0 U \rightarrow T_{{\phi} (0)} {\phi} (U)
\]
is an isomorphism of vectorial spaces. Hence ${d} ({^t {\phi} _1}),\ldots ,{d} ({^t {\phi} _n})$ are linearly independent over $M_{\mathbb{R},n}$.
In particular, by Lemma \ref{dindependence}.$(2)$, ${d} \psi _1,\ldots ,{d} \psi _n$ are also linearly independent over $M_{\mathbb{R},n}$.

Consider the real meromorphic function
$$f :\mathbb{C}^n\rightarrow \mathbb{C}^n:u\mapsto f(u):={^a\psi(u)}$$
which admits an AAT by definition. We first show that there exists $c\in \mathbb{R}^n$ and an open neighborhood $U'\subset \mathbb{R}^n$ of $0$ such that
\[
\varphi:U'\rightarrow \mathbb{R}^n: u\mapsto \varphi(u):=f(u+c)
\]
is an analytic diffeomorphism onto its image, so we will have a locally Nash group structure $(\mathbb{R}^n,+,f)$ and shrinking $U'$ if neccesary we may assume that $(U',\varphi)$ is one of its charts.
Indeed,  by Lemma \ref{pole}.$(1)$ there exists an open dense subset 
$W\subset \mathbb{R}^n$ such that 
\[
f|_W:W\rightarrow \mathbb{R}^n
\]
is analytic. Let $\Delta$ denote the determinant of the Jacobian of $f|_W$.
Since ${d} \psi _1,\ldots ,{d} \psi _n$ are linearly independent over $M_{\mathbb{R},n}$, 
$\Delta$ is not identically zero on $W$ and hence there exists $c\in U\cap W$ such that $\Delta (c)\neq 0$. 
So $f$ is a local diffeomorphism at $c$ and hence there exists an open neighborhood $U'\subset \mathbb{R}^n$ of $0$ such that
$\varphi:U'\rightarrow \mathbb{R}^n: u\mapsto \varphi(u):=f(u+c)$ is a diffeomorphism onto its image, as required.

Finally, by Proposition \ref{algebraic for functions} we have that $f$ is algebraic over $\mathbb{C}(f(u+k))$ and therefore $\phi$ is algebraic over $\mathbb{R}(\varphi)$ on a sufficiently small open neighborhood of $0$, so that by Proposition \ref{compatibilityAut} the 
identity map from $(\mathbb{R}^n,+,\phi|_U )$ to $(\mathbb{R}^n,+,\, \varphi|_{U'} )$ is a locally Nash isomorphism.\end{proof}

We point out that Proposition \ref{compatibilityAut} gives a criterion to decide whether two locally Nash structures on $(\mathbb{R}^n,+)$ are locally Nash isomorphic.

\begin{corollary}\label{compatibilityAut.1}
Let $(\mathbb{R}^n,+, f)$ and $(\mathbb{R}^n,+, g)$ be locally Nash groups, where $f,g:\mathbb{C}^n\rightarrow \mathbb{C}^n$ are 
real meromorphic maps admitting an AAT.
Then, they are isomorphic as locally Nash groups if and only if there exists $\alpha \in GL_n(\mathbb{R})$ such that 
$g\circ \alpha$ is algebraic over $\mathbb{R}(f)$.  
\end{corollary}
\begin{proof}
By hypothesis there exist $k_1\in \mathbb{R}^n$ and an open neighborhood of the identity $U$ of $\mathbb{R}^n$ such that 
$(\mathbb{R}^n,+, f)$ denotes $(\mathbb{R}^n,+,\phi |_U)$ where 
\[
\phi : U \rightarrow \mathbb{R}^n : u\mapsto f(u+k_1).
\]
Similarly there exist $k_2\in \mathbb{R}^n$ and an open neighborhood of the identity $V$ of $\mathbb{R}^n$ such that 
$(\mathbb{R}^n,+, g)$ denotes $(\mathbb{R}^n,+,\psi|_V)$ where 
\[
\psi : V\rightarrow \mathbb{R}^n : u\mapsto g(u+k_2).
\]
By Corollary \ref{algebraic for functions}.$(1)$ we have that $f_{u+k_1}:=f(u+k_1)$ is algebraic over $\mathbb{C}(f(u))$ and the other way around, and similarly for   $g_{u+k_2}:=g(u+k_2)$ and $g(u)$. In particular, for any $\alpha \in GL_n(\mathbb{R})$ we have that $g\circ \alpha$ is algebraic over $\mathbb{R}(f)$ if and only if 
 $g_{u+k_2}\circ \alpha$ is algebraic over $\mathbb{R}(f_{u+k_1})$.
 
We suppose first that $\alpha$ is an isomorphism of locally Nash groups
\[
\alpha : (\mathbb{R}^n,+, \phi |_U) \rightarrow (\mathbb{R}^n,+, \psi |_V). 
\]
Note that $\alpha \in GL_n(\mathbb{R})$.
Applying Proposition \ref{compatibilityAut} there exists $W\subset U\cap \alpha ^{-1}(V)$ such that
$\psi \circ \alpha$ is algebraic over $\mathbb{R}(\phi )$ on $W$. We deduce that $g_{u+k_2}\circ \alpha$ is algebraic over $\mathbb{R}(f_{u+k_1})$ and therefore $g\circ \alpha$ is algebraic over $\mathbb{R}(f)$.

We show the right to left implication.
Since  $g\circ \alpha$ is algebraic over $\mathbb{R}(f)$, it follows that $g_{u+k_2}\circ \alpha$ is algebraic over $\mathbb{R}(f_{u+k_1})$. Therefore $\psi \circ \alpha$ is algebraic over $\mathbb{R}(\psi)$ on a sufficiently small neighborhood of $0$. Finally, since $\alpha$ is a continuous isomorphism, we apply  Proposition \ref{compatibilityAut} to $\alpha$ and we obtain that 
$(\mathbb{R}^n,+, \phi |_U)$ and $(\mathbb{R}^n,+, \psi |_V)$ are isomorphic as locally Nash groups.
\end{proof}

\medskip

\subsection{Periods of meromorphic maps.} Finally, we introduce some invariants that will allow us to describe Nash atlas of $(\mathbb{R}^n,+)$.
Let $\Lambda$ be a discrete subgroup of $(\mathbb{C}^n,+)$.
Then there exist $r\leq 2n$ and $\lambda _1,\ldots ,\lambda _r\in \Lambda$ linearly independent over $\mathbb{R}$ such that 
\[
\Lambda = \mathbb{Z}\lambda _1 \oplus \ldots \oplus \mathbb{Z}\lambda _r.
\]
We call $r$ (the dimension of $\Lambda$ as a (free) $\mathbb{Z}$-module) the rank of $\Lambda$, which is independent of the chosen 
basis, and denote it $rank\, \Lambda$.
A discrete subgroup $\Lambda$ of $(\mathbb{C}^n,+)$ is a \emph{lattice} if $rank\, \Lambda =2n$.
We say that a subgroup $G<(\mathbb{C}^n,+)$ is a \emph{real subgroup} if $G =\overline{G}$.
We say $\Lambda$ is a \emph{real discrete subgroup} (resp. \emph{real lattice}) of $(\mathbb{C}^n,+)$ if $\Lambda$ is both 
a discrete subgroup (resp. lattice) of $(\mathbb{C}^n,+)$ and a real subgroup of $(\mathbb{C}^n,+)$.

The previous concepts are related to meromorphic maps as follows.
Given a meromorphic map $f:\mathbb{C}^n\rightarrow \mathbb{C}^m$, we define the \emph{group of periods of $f$} as
\[
\Lambda _f:=\{  a \in \mathbb{C}^n : f(u)=f(u+a)\}
\]
where $f(u)=f(u+a)$ means that if $f=g/h$ then $g(b)h(b+a)=h(b)g(b+a)$ for all $b\in \mathbb{C}^n$.
Note that $\Lambda _f$ is a subgroup of $(\mathbb{C}^n,+)$ that may not be discrete.
However, we have the following:

\begin{lemma}\label{discrete groups}
Let $f:\mathbb{C}^n\rightarrow \mathbb{C}^n$ be a meromorphic map.
\begin{enumerate}
\item[$(1)$] If $f$ is a local diffeomorphism at $0$ 
then $\Lambda _f$ is a discrete subgroup of $(\mathbb{C}^n,+)$.
\item[$(2)$] If $f$ is a real meromorphic map then $\Lambda _f$ is a real subgroup of $(\mathbb{C}^n,+)$.
\item[$(3)$] If $f$ is a real meromorphic map, $k\in \mathbb{R}^n$ and for some open neighborhood of the identity 
$U\subset \mathbb{R}^n$ the restriction of $f(u+k)$ to $U$ is an analytic diffeomorphism then $\Lambda _f$ 
is a real discrete subgroup of $(\mathbb{C}^n,+)$.
\item[$(4)$] If $\Lambda _f$ is a discrete subgroup of $(\mathbb{C}^n,+)$ and $\alpha \in GL_n(\mathbb{C})$ then 
$\Lambda _{f\circ \alpha}$ is a discrete subgroup of $(\mathbb{C}^n,+)$ with $rank\, \Lambda _{f\circ \alpha}=rank\, \Lambda _f$.
\end{enumerate}
\end{lemma}
\begin{proof}
$(1)$ Clearly $\Lambda _f$ is a subgroup of $(\mathbb{C}^n,+)$.
Suppose for a contradiction that $\Lambda _f$ is not discrete.
Then there exists an infinite sequence $\{ a_i:i\in \mathbb{N}\}$ of points of $\Lambda _f$ that converges to some 
$a\in \mathbb{C}^n$.
Take $r>0$ such that $f$ is injective and analytic in an open ball of radius $r$ centered at $0$, this can be done because $f$ is
a local diffeomorphism at $0$.
Take $N\in \mathbb{N}$ such that $\| a_i-a_N\| <r$ for all $i\geq N$.
Since $\Lambda _f$ is a subgroup of $(\mathbb{C}^n,+)$, $a_i-a_N\in \Lambda _f$ for all $i\in \mathbb{N}$.
This implies that $f(a_i-a_N)=f(0)$ for all $i\in \mathbb{N}$, which contradicts that $f$ is injective in the ball of radius $r$ 
centered at $0$.

$(2)$ Fix $\lambda \in \Lambda _f$.
By definition $f(u)=f(u+\lambda)$.
Hence, $f(\overline{u})=f(\overline{u}+\overline{\lambda})$ because $f$ is a real meromorphic function.
Therefore, $\overline{\lambda}\in \Lambda _f$.

$(3)$ We may assume that $k=0$.
Let $J$ be the determinant of the Jacobian of $f|_U$ at $0$.
Since $f^{-1}|_{f(U)}$ exists, $J\neq 0$.
Since the determinant of the Jacobian of $f$ at $0$ is also $J$, it is not $0$.
So by the inverse mapping theorem $f$ is a local diffeomorphism at $0$.
Hence by $(1)$ and $(2)$, $\Lambda _f$ is a real discrete subgroup of $(\mathbb{C}^n,+)$.

$(4)$ Take $r\leq 2n$ and $\lambda _1,\ldots ,\lambda _r\in \Lambda$ linearly independent over $\mathbb{R}$ such that 
\[
\Lambda _f= \mathbb{Z}\lambda _1 \oplus \ldots \oplus \mathbb{Z}\lambda _r.
\]
Then, $\alpha ^{-1}(\lambda _1),\ldots ,\alpha ^{-1}(\lambda _r)\in \Lambda$ are linearly independent over $\mathbb{R}$ and
\[
\Lambda _{f\circ \alpha}= \mathbb{Z}\alpha ^{-1}(\lambda _1) \oplus \ldots \oplus \mathbb{Z}\alpha ^{-1}(\lambda _r).
\]
\end{proof}

By Theorem \ref{T1} every locally Nash group structure on $(\mathbb{R}^n,+)$ is of the form $(\mathbb{R}^n,+,f)$ for a certain real meromorphic map $f:\mathbb{C}^n\rightarrow \mathbb{C}^n$ admitting an AAT. In Proposition \ref{different ranks} we will show  that $rank\, \Lambda _f$ is an invariant of the locally Nash isomorphism class. We first prove a technical lemma.

\begin{lemma}\label{algebraicity of periods} Let $f,g:\mathbb{C}^n\rightarrow \mathbb{C}^n$ be meromorphic maps such that 
$\Lambda _f$ and $\Lambda _g$ are discrete subgroups of $(\mathbb{C}^n,+)$.
If $g$ is algebraic over $\mathbb{C}(f)$ then there exists $a\in \mathbb{N}\setminus \{0\}$ such that $a\Lambda _f<\Lambda _g$, 
and in particular $rank\, \Lambda _f \leq rank\, \Lambda _g$.
Furthermore, if $g_1,\ldots ,g_n$ are algebraically independent over $\mathbb{C}$ then $rank\, \Lambda _f = rank\, \Lambda _g$.
\end{lemma}
\begin{proof}
We prove the first clause.
We may assume that $\Lambda _f\neq \{0\}$.
Take $\lambda \in \Lambda _f\setminus \{0\}$ and fix $j\in \{1,\ldots ,n\}$.
Let $P_j(Z)$ be minimum polynomial of $g_j(u)$ over $\mathbb{C}(f(u))$.
Since $\lambda \in \Lambda _f$, $P_j(g_j(u+\ell \lambda))\equiv 0$ for each $\ell\in \mathbb{Z}$.
Since $P_j(Z)$ has a finite number of roots, there exist $\ell_1,\ell_2\in \mathbb{Z}$, $\ell_2>\ell_1$, such that 
$g_j(u+\ell_1\lambda )=g_j(u+\ell_2\lambda)$.
Let $a_j:=\ell_2-\ell_1\in \mathbb{N}\setminus \{0\}$.
Then $g_j(u)=g_j(u+a_j\lambda )$ and hence $g_j(u)=g_j(u+\ell a_j\lambda )$ for each $\ell \in \mathbb{Z}$.
Let $a$ be the least common multiple of $a_1,\ldots ,a_n$.
Then $g_j(u)=g_j(u+\ell a\lambda )$ for each $\ell\in \mathbb{Z}$ and each $j\in \{1,\ldots ,n\}$, so $a\lambda \in \Lambda _g$.
Let now $\{\lambda_1,\ldots ,\lambda_m\}$ be a basis for $\Lambda _f$.
Take $a$ again be the l.c.m. of the $a$'s such that $a\lambda _i\in \Lambda _g$.
Then, for this $a$ we have $a\lambda \in \Lambda _g$ for each $\lambda\in \Lambda _f$.
This also shows that $\Lambda _g$ contains at least $rank\, \Lambda _f$ linearly independent vectors over $\mathbb{R}$ and hence 
$rank\,\Lambda _f \leq rank\, \Lambda _g$.

\smallskip

The other clause follows by symmetry since if $g_1,\ldots ,g_n$ are algebraically independent over $\mathbb{C}$
then $f$ is algebraic over $\mathbb{C}(g)$.
\end{proof}

The next corollary of Lemma \ref{algebraicity of periods} will be useful to study Weierstrass $\wp$-functions in the context of the one-dimensional classification of locally Nash groups.

\begin{corollary}\label{algebraicity of periods2} Let $f,g:\mathbb{C}^n\rightarrow \mathbb{C}^n$ be meromorphic maps such that both 
$\Lambda _f$ and $\Lambda _g$ are discrete subgroups of $\mathbb{C}^n$.
If $g$ is algebraic over $\mathbb{C}(f)$ then there exists a discrete subgroup $\Lambda$ of $(\mathbb{C}^n,+)$ such that 
$rank\, \Lambda =rank\, \Lambda _f$ and both $\Lambda <\Lambda _f$ and $\Lambda <\Lambda _g$.
Furthermore, if $\Lambda _f$ is a real discrete subgroup then we can take $\Lambda$ to be a real discrete subgroup.
\end{corollary}
\begin{proof}
By Lemma \ref{algebraicity of periods} there exists $a\in \mathbb{N}\setminus \{0\}$ such that $a\Lambda _f<\Lambda _g$.
It suffices to take $\Lambda =a\Lambda _f$.
\end{proof}

\begin{proposition}\label{different ranks}
Let $(\mathbb{R}^n,+,f)$ and $(\mathbb{R}^n,+,g)$ be locally Nash groups, where $f,g:\mathbb{C}^n\rightarrow \mathbb{C}^n$ 
are real meromorphic maps admitting an AAT.
If $(\mathbb{R}^n,+,f)$ and $(\mathbb{R}^n,+,g)$ are isomorphic as locally Nash groups then 
$rank\, \Lambda _f=rank\, \Lambda _g$.
\end{proposition}
\begin{proof}
By Corollary \ref{compatibilityAut.1} there exists $\alpha \in GL_n(\mathbb{R})$ such that $g\circ \alpha$ is algebraic 
over $\mathbb{R}(f)$.
We note that by Lemma \ref{discrete groups}.$(3)$ both $\Lambda _g$ and $\Lambda _f$ are real discrete subgroups of 
$(\mathbb{C}^n,+)$.
By Lemma \ref{discrete groups}.$(4)$, $\Lambda _{g\circ \alpha }$ is also a discrete subgroup of $(\mathbb{C}^n,+)$,
with $rank\, \Lambda _{g\circ \alpha}=rank\, \Lambda _g$.
Now, by Lemma \ref{algebraicity of periods}, $rank\, \Lambda _{g\circ \alpha}=rank\, \Lambda _f$ and hence 
$rank\, \Lambda _f=rank\, \Lambda _g$.
\end{proof}

Lemma \ref{algebraicity of periods} leads us to the the following definition.
Let $\mathbb{L}$ be a field of meromorphic functions from $\mathbb{C}^n$ to $\mathbb{C}$ of transcendence degree 
$n$ over $\mathbb{C}$.
Suppose that there exists $f:=(f_1,\ldots ,f_n):\mathbb{C}^n\rightarrow \mathbb{C}^n$ such that $\{ f_1,\ldots ,f_n\}$ is a 
transcendence basis of $\mathbb{L}$ over $\mathbb{C}$ and $\Lambda _f$ is a discrete subgroup of $(\mathbb{C}^n,+)$.
Then, by Lemma \ref{algebraicity of periods}, for all $g:=(g_1,\ldots ,g_n):\mathbb{C}^n\rightarrow \mathbb{C}^n$ such that 
$\{ g_1,\ldots ,g_n\}$ is a transcendence basis of $\mathbb{L}$ over $\mathbb{C}$ we have that $\Lambda _g$ is a discrete 
subgroup of $(\mathbb{C},^n)$ with $rank \,\Lambda _g=rank \,\Lambda _f$.
Hence, we introduce the following notation, that will be useful in the proof of Theorem \ref{T3}.

\begin{definition}\label{Z-rank}
\emph{
Let $\mathbb{L}$ be a field of meromorphic functions from $\mathbb{C}^n$ to $\mathbb{C}$ of transcendence degree 
$n$ over $\mathbb{C}$.
Suppose that there exists $f:=(f_1,\ldots ,f_n):\mathbb{C}^n\rightarrow \mathbb{C}^n$ such that $\{ f_1,\ldots ,f_n\}$ is a 
transcendence basis of $\mathbb{L}$ over $\mathbb{C}$ and $\Lambda _f$ is a discrete subgroup of $(\mathbb{C}^n,+)$.
Then, we say that $\mathbb{Z}$--$rank\, \mathbb{L}=m$ if $rank\,  \Lambda _f=m$.
Otherwise, we say that the $\mathbb{Z}$--$rank$ of $\mathbb{L}$ is not defined.
Let $\mathcal{P}=\{ \mathbb{L}_{\gamma}: \gamma \in \Gamma  \}$ where each $\mathbb{L}_{\gamma}$ is a field of meromorphic 
functions from $\mathbb{C}^n$ to $\mathbb{C}$ of transcendence degree $n$ over $\mathbb{C}$.
We say that $\mathbb{Z}$--$rank\, \mathcal{P}=m$ if $\mathbb{Z}$--$rank\, \mathbb{L}_{\gamma}=m$ for every $\gamma\in \Gamma$.
}
\end{definition}

\section{Two-dimensional simply connected abelian locally Nash groups.}\label{section2d}

In this section we will give a description of the two-dimensional simply connected abelian locally Nash groups, 
which is based on a theorem of Painlev\'e published in \cite{Painleve}.
Since Painlev\'e wrote \cite{Painleve} in $1902$, some of its notation is outdated.
We proceed to introduce and clarify its notation.

\medskip

For Painlev\'e a meromorphic map $f:\mathbb{C}^n\rightarrow \mathbb{C}^n$ admits an algebraic addition theorem if and only if the 
coordinate functions of $f$ are functionally independent and $f$ admits an AAT (in our sense). Any $n$ functions are functionally independent if {\itshape ils ne sont li\'ees par aucune relation identique}, see the footstep 
note of the first page of \cite{Painleve}. With this definition Painlev\'e is refering to the classical functional independence, 
see for example \cite[Definition 3]{Newns} for a detailed treatment.
Another characterization  of functional independence which will be more convenient for our purposes is the following 
(see \cite[Proposition 1]{Newns}). Let $\mathbb{K}$ be $\mathbb{R}$ or $\mathbb{C}$, we say that $f_1,\ldots ,f_n:\mathbb{K}^n\rightarrow \mathbb{K}$ 
are \emph{functionally independent} if the range of  $f :=(f_1,\ldots ,f_n):\mathbb{K}^n\rightarrow \mathbb{K}^n$ 
has an interior point in $\mathbb{K}^n$. We will apply Painlev\'e's results to meromorphic maps associated to translations of charts of the identity of locally Nash groups, which are  clearly functionally independent.

\medskip

To state Painlev\'e's results we introduce the following notation (see also \cite[Ch. 5 \textsection 6]{Siegel}). Recall that a meromorphic function $f$ is \emph{degenerate} if $\Lambda _f$ is not a discrete subgroup of $\mathbb{C}^n$. Let $\Lambda$ be a lattice of $(\mathbb{C}^n,+)$.
We say a meromorphic function $f:\mathbb{C}^n\rightarrow \mathbb{C}$ is an \emph{abelian function corresponding 
to $\Lambda$} if $\Lambda _f>\Lambda$. The abelian functions corresponding to $\Lambda$ form a field that we denote $\mathbb{C}(\Lambda )$.
Clearly $\mathbb{C}(\Lambda)$ contains degenerate functions, for example all the constants.
We say $\mathbb{C}(\Lambda )$ is \emph{nondegenerate} if it contains at least one function that is not degenerate.
We note that depending on $\Lambda$, the transcendence degree of $\mathbb{C}(\Lambda )$ can be from $0$ to $n$.
However, $\mathbb{C}(\Lambda )$ is nondegenerate if and only if its transcendence degree over $\mathbb{C}$ is $n$ (see,
\cite[Ch. 5 \textsection 11 Theorems 5 and 6]{Siegel}).
Given a lattice $\Omega$ of $(\mathbb{C},+)$ we will consider the Weierstrass functions $\wp _{\Omega}$, $\sigma _{\Omega}$ and 
$\zeta _{\Omega}$ (see e.g \cite[Ch.III and IV]{Chandrasekharan}). 
Recall that
$$\begin{array}{lcl}
\sigma _{\Omega}(u) & = & u\prod _{\omega \in \Omega \setminus \{0\}}, 
\left( 1-\frac{u}{\omega}\right)\exp\left(\frac{u}{\omega}+\frac{1}{2}\left(\frac{u}{\omega}\right)^2\right), \vspace{1mm}\\
\zeta _{\Omega}(u) & = & \frac{\sigma '_{\Omega}(u)}{\sigma _{\Omega}(u)}, \vspace{1mm}\\
\wp _{\Omega}(u) & = & -\zeta '_{\Omega}(u).
\end{array}$$

Finally, we define the \emph{families the Painlev\'e's description} as follows:
\begin{enumerate}
 \item[] $\mathcal{P}_1 :=\{ \, \mathbb{C}(g_1\circ \alpha ) : \alpha \in GL_2(\mathbb{C})\, \}, \text{ where } g_1(u,v):=(u,v);$ \vspace{1mm}
 \item[] $\mathcal{P}_2 :=\{ \, \mathbb{C}(g_2\circ \alpha ) : \alpha \in GL_2(\mathbb{C})\, \}, \text{ where } g_2(u,v):=(e^u,v);$ \vspace{1mm}
 \item[] $\mathcal{P}_3 :=\{ \, \mathbb{C}(g_3\circ \alpha ) : \alpha \in GL_2(\mathbb{C})\, \}, \text{ where } g_3(u,v):=(e^u,e^v);$ \vspace{1mm}
 \item[] $\mathcal{P}_4 :=\{ \, \mathbb{C}(g_{4,a,\Omega }\circ \alpha ) : \alpha \in GL_2(\mathbb{C}),\, a \in \{0,1\},
\, \Omega \text{ is a lattice of } (\mathbb{C},+)\, \}$ 

where $g_{4,a,\Lambda}(u,v)=(\wp _{\Omega }(u),v-a \zeta _{\Omega } (u))$; \vspace{1mm}
  \item[] $\mathcal{P}_5 :=\{ \, \mathbb{C}(g_{5,a,\Omega}\circ \alpha ) : \alpha \in GL_2(\mathbb{C}),\, a \in \mathbb{C}, 
\,\Omega \text{ is a lattice of } (\mathbb{C},+)\, \}$ 

where $g_{5,a,\Omega }(u,v)=\left(\wp _{\Omega}(u),\frac{\sigma _\Omega (u-a)}{\sigma _\Omega (u)}e^v\right)$; and \vspace{1mm}
 \item[] $\mathcal{P}_6:=\{  \, \mathbb{C}(\Lambda ) : \Lambda \text{ is a lattice of } (\mathbb{C}^2,+),\,
tr.deg.\, _{\mathbb{C}}\mathbb{C}(\Lambda )=2 \}.$
\end{enumerate}

It can be checked that $g_{4,a,\Omega}$ is algebraic over $\mathbb{C}(g_{4,1,\Omega})$ for each $a\neq 0$, this is the reason why only
$a\in \{0,1\}$ are considered in the family $\mathcal{P}_4$.
Henceforth we keep the notation $g_1$, $g_2$, $g_3$, $g_{4,a,\Omega}$ and $g_{5,a,\Omega}$ exclusively for these mentioned functions.
In Theorem \ref{T3} we will show that these maps admit an AAT and therefore induce a locally Nash structure on $(\mathbb{R}^2,+)$.

\medskip

Now we can state the main result of \cite{Painleve}.
\begin{fact}[{Painlev\'e, \cite[Main Theorem]{Painleve}}]\label{TP}
If $f_1,f_2:\mathbb{C}^2\rightarrow \mathbb{C}$ are functionally independent meromorphic functions such that $f:=(f_1,f_2)$
admits an AAT then there exist $i\in \{1,\ldots ,6\}$ such that $f_1(u,v)$ and $f_2(u,v)$ are algebraic over one of the fields 
of the family $\mathcal{P}_i$.
\end{fact}

In the next lemma we analyze the group of periods of the  families of Painlev\'e's theorem.
Firstly, we will list the properties of the Weierstrass $\sigma$ and $\zeta$ functions that will be needed.

\begin{fact}\emph{(\cite[Ch.IV]{Chandrasekharan})}\label{properties}
Let $\Omega :=<\omega _1,\omega _2>_{\mathbb{Z}}$ be a lattice of $\mathbb{C}$.
Then,
\begin{enumerate}
\item[$(1)$] \quad $\displaystyle \zeta _\Omega (z+\omega _i)=\zeta _\Omega (z)+2\zeta _\Omega (\omega _i/2)$ 
\quad for each $i\in \{1,2\}$,
\item[$(2)$] \quad $\displaystyle 
\sigma _\Omega (z+\omega _i)=-\sigma _\Omega(z)e^{2\zeta _\Omega (\frac{\omega _i}{2})(z+\frac{\omega _i}{2})}$ 
\quad for each $i\in \{1,2\}$,
\end{enumerate}
\end{fact}

\begin{lemma}\label{periods} Let $a \in \mathbb{C}$ and $\Omega :=<\omega _1,\omega _2>_{\mathbb{Z}}$ 
be a lattice of $(\mathbb{C},+)$.
Then, 
\begin{enumerate}
\item[$(1)$] \quad $\displaystyle \Lambda _{g_1}=\{(0,0)\}$,
\item[$(2)$] \quad $\displaystyle \Lambda _{g_2}=<(2\pi i,0)>_{\mathbb{Z}}$,
\item[$(3)$] \quad $\displaystyle \Lambda _{g_3}=< (2\pi i,0), (0,2\pi i)>_{\mathbb{Z}}$,
\item[$(4)$] \quad $\displaystyle \Lambda _{g_{4,a, \Omega}}=<(\omega _1 ,2a \zeta _\Omega (\omega _1/2)),
(\omega _2 ,2a \zeta _\Omega (\omega _2/2)) >_{\mathbb{Z}}$,
\item[$(5)$] \quad $\displaystyle \Lambda _{g_{5,a ,\Omega}}=<(\omega _1 ,2a \zeta _\Omega (\omega _1/2)),
(\omega _2 ,2a \zeta _\Omega (\omega _2/2),(0,2\pi i)) >_{\mathbb{Z}}$.
\end{enumerate}
\end{lemma}
\begin{proof}The only non trivial cases are the last two ones when $a\neq 0$. On the other hand, it is easy to check using Fact \ref{properties} that the above tuples are periods of the corresponding map.

We begin with the case $g_{4,a,\Omega}$.
Let $g$ denote $g_{4,a,\Omega}$ and $g_1$ and $g_2$ denote the coordinate functions of $g$.
Fix $\lambda :=(\lambda _1,\lambda _2)\in \Lambda _g$.
Clearly $\Lambda _g\subset \Lambda _{g_1}\cap \Lambda _{g_2}$.
Since $\lambda \in \Lambda _{g_1}$, we have that $\lambda _1\in \Omega$.
Fix $m,n\in \mathbb{Z}$ such that $\lambda _1 =m\omega _1+n\omega _2$.
It follows from Fact \ref{properties}.(1) that
\[
\zeta _\Omega (u+m\omega _1+n\omega _2) - \zeta _\Omega (u)=2m\zeta _\Omega (\omega _1/2)+2n\zeta _\Omega (\omega _2/2).
\tag{$\ast$}\label{Chfact}
\]
Since $\lambda \in \Lambda _{g_2}$, we have that 
\[
v+\lambda _2-a\zeta _\Omega (u+m\omega _1+n\omega _2) = v-a\zeta _\Omega (u)
\]
and hence by equation (\ref{Chfact}) 
\[
\lambda _2=2am\zeta _\Omega (\omega _1/2)+2an\zeta _\Omega (\omega _2/2).
\]
This means that the elements of $\Lambda _g$ are of the form 
\[
\big(m\omega _1+n\omega _2, 2am\zeta _\Omega (\omega _1/2)+2an\zeta _\Omega (\omega _2/2)\big)
\]
with $m,n\in \mathbb{Z}$, so we are done with this case.

Now we show the case $g_{5,a,\Omega}$.
Let $g$ denote $g_{5,a,\Omega}$ and $g_1$ and $g_2$ denote the coordinate functions of $g$.
Fix $\lambda :=(\lambda _1,\lambda _2)\in \Lambda _g$.
Reasoning as before we get that there exists $m,n\in \mathbb{Z}$ such that $\lambda _1 =m\omega _1+n\omega _2$.
Moreover, again by Fact \ref{properties}.(2) and from equation (\ref{Chfact}) we get that
\[
\frac{\sigma _\Omega (u+m\omega _1+n\omega _2)}{\sigma _\Omega (u)}=
Ce^{u(2m\zeta _\Omega (\omega _1/2)+2n\zeta _\Omega (\omega _2/2))}\tag{$\dagger$}\label{Chfact2}
\]
for some constant $C\in \mathbb{C}$.
Since $\lambda \in \Lambda _{g_2}$, we have that
\[
\frac{\sigma _\Omega (u+\lambda _1-a)}{\sigma _\Omega (u+\lambda _1)}e^{v+\lambda _2} 
=
\frac{\sigma _\Omega (u-a)}{\sigma _\Omega (u)}e^v
\]
and hence 
\[
e^{\lambda _2}=\frac{\sigma _\Omega (u-a)}{\sigma _\Omega (u)}\frac{\sigma _\Omega (u+\lambda _1)}{\sigma _\Omega (u+\lambda _1-a)}.
\]
So by equation (\ref{Chfact2}) we get
\[
e^{\lambda _2}=e^{2a(m\zeta _\Omega (\omega _1/2)+n\zeta _\Omega (\omega _2/2))}
\]
and hence
\[
\lambda _2=2am\zeta _\Omega (\omega _1/2)+2an\zeta _\Omega (\omega _2/2)+2\ell \pi i
\]
for some $\ell \in \mathbb{Z}$.
This means that the elements of $\Lambda _g$ are of the form 
\[
\big(m\omega _1+n\omega _2, 2am\zeta _\Omega (\omega _1/2)+2an\zeta _\Omega (\omega _2/2)+2\ell \pi i\big)
\]
with $\ell,m,n\in \mathbb{Z}$, which concludes the proof.	
\end{proof}

Now we study the $\mathbb{Z}$--$rank$ of the Painlev\'e's families of fields (the $\mathbb{Z}$--$rank$ was 
defined in Definition \ref{Z-rank}).

\begin{proposition}\label{rank} The $\mathbb{Z}$--$rank$s of the Painlev\'e's families are
$\mathbb{Z}$--$rank\, \mathcal{P}_1=0$, 
$\mathbb{Z}$--$rank\, \mathcal{P}_2=1$, 
$\mathbb{Z}$--$rank\, \mathcal{P}_3=2$,
$\mathbb{Z}$--$rank\, \mathcal{P}_4=2$, 
$\mathbb{Z}$--$rank\, \mathcal{P}_5=3$ and 
$\mathbb{Z}$--$rank\, \mathcal{P}_6=4$.
\end{proposition}
\begin{proof}
Let $i\in \{1,2,3\}$.
Firstly, note that since $\alpha\in GL_2(\mathbb{C})$ by Lemma \ref{discrete groups}.$(4)$ the fields belonging to the same 
family $\mathcal{P}_i$ have the same $\mathbb{Z}$--$rank$, which is $rank\, \Lambda _{g_i}$.
Then apply Lemma \ref{periods} to deduce that $rank\, \Lambda _{g_i}=i-1$.

Let $i\in \{4,5\}$.
As above, it suffices to consider $rank\, \Lambda _{g_i,a,\Omega}$.
By Lemma \ref{periods} these ranks are independent of $a$ and $\Omega$ and hence $rank\, \Lambda _{g_i,a,\Omega}=i-2$.

Finally, we consider the case of the abelian functions. 
Let $\Lambda$ be a lattice of $(\mathbb{C}^2,+)$ such that $\mathbb{C}(\Lambda )$ has transcendence degree $2$ over $\mathbb{C}$. Fix a transcendence basis $\{f_1,f_2\}$ of $\mathbb{C}(\Lambda)$ and let us see that $rank \Lambda_f=4$. By definition $\Lambda < \Lambda_f$ and therefore it is enough to check that $\Lambda_f$ is discrete. Since $\text{tr.deg.}_\mathbb{C}\mathbb{C}(\Lambda)=2$, there exist a nondegenerate meromorphic function $g\in \mathbb{C}(\Lambda)$. In particular, $g$ is algebraic over $\mathbb{C}(f)$. Arguing as in the proof of Lemma \ref{algebraicity of periods}, if $\Lambda_f$ is not discrete then $\Lambda_g$ is not discrete, a contradiction.\end{proof}

Finally, we consider the possible locally Nash group structures over $(\mathbb{R}^2,+)$ induced by Painlev\'e's description
(Fact \ref{TP}).

\begin{theorem}\label{T3}
Every simply connected $n$-dimensional abelian locally Nash group is locally Nash isomorphic to one of the form 
$(\mathbb{R}^2,+,f)$ where $f:\mathbb{C}^2\rightarrow \mathbb{C}^2$ is a real meromorphic map admitting an AAT and such 
that its coordinate functions are algebraic over one of the fields of the following families:
\begin{enumerate}
\item[$(1)$] $\mathcal{P}_1 :=\{ \, \mathbb{C}(g_1\circ \alpha ) : \alpha \in GL_2(\mathbb{C})\, \}$, where $g_1(u,v)=(u,v)$; \vspace{1mm}
\item[$(2)$] $\mathcal{P}_2 :=\{ \, \mathbb{C}(g_2\circ \alpha ) : \alpha \in GL_2(\mathbb{C})\, \}$, where $g_2(u,v)=(u,e^v)$; \vspace{1mm}
\item[$(3)$] $\mathcal{P}_3 :=\{ \, \mathbb{C}(g_3\circ \alpha ) : \alpha \in GL_2(\mathbb{C})\, \}$, where $g_3(u,v)=(e^u,e^v)$; \vspace{1mm}
\item[$(4)$] $\mathcal{P}_4 :=\{ \, \mathbb{C}(g_{4 ,a, \Omega}\circ \alpha ) : \alpha \in GL_2(\mathbb{C}),\, a \in \{0,1\}, 
\,\Omega \text{ is a lattice of } (\mathbb{C},+)\, \}$, where $g_{4,a ,\Omega}(u,v)=(\wp _{\Omega}(u),v-a\zeta _{\Omega} (u))$; \vspace{1mm}
\item[$(5)$] $\mathcal{P}_5 :=\{ \, \mathbb{C}(g_{5,a ,\Omega}\circ \alpha ) : \alpha \in GL_2(\mathbb{C}),\, a \in \mathbb{C}, 
\,\Omega \text{ is a lattice of } (\mathbb{C},+)\, \}$, where $g_{5,a ,\Omega}(u,v)=
\left(\wp _{\Omega}(u),\frac{\sigma _{\Omega} (u-a)}{\sigma _{\Omega} (u)}e^v\right)$; and \vspace{1mm}
\item[$(6)$] $\mathcal{P}_6:=\{  \, \mathbb{C}(\Lambda ) : \Lambda \text{ is a lattice of } (\mathbb{C}^2,+),\,
tr.deg.\, _{\mathbb{C}}\mathbb{C}(\Lambda )=2 \}$.
\end{enumerate}
Furthermore, if $(\mathbb{R}^2,+,g)$ is a locally Nash group, where $g:\mathbb{C}^2\rightarrow \mathbb{C}^2$ is a real meromorphic 
map admitting an AAT, and the coordinate functions of $f$ and $g$ are algebraic over fields of different families, then $(\mathbb{R}^2,+,f)$ and 
$(\mathbb{R}^2,+,g)$ are not locally Nash isomorphic.

Even further, each of the families induce at least one locally Nash group structure on $(\mathbb{R}^2,+)$.
\end{theorem}
\begin{proof}
Firstly, we point out that $g_1,g_2,g_3$ clearly admit an AAT,  $g_{4,a,\Omega}$ and $g_{5,a ,\Omega}$ admit it by
\cite[Art. 16 and 19]{Painleve} and the fact that $g_{4,a,\Omega}$ is algebraic over $g_{4,1,\Omega}$ for all $a\neq 0$. On the other hand, any transcendence basis of a field in  $\mathcal{P}_6$ satisfies an AAT by \cite[Chap 5. \textsection 13]{Siegel}.
 
We show that each family induce at least one locally Nash group structure on $(\mathbb{R}^2,+$). Fix $a\in \mathbb{R}$ and a real lattice $\Omega$ of $(\mathbb{C},+)$. We recall that then $g_1$, \ldots , $g_{5,\alpha ,\Omega}$ are real meromorphic maps, each one admitting an AAT.
Also, since $\wp_\Omega$ admits an AAT, $g_6(u,v):=(\wp _\Omega (u),\wp _\Omega (v))$ is a real meromorphic map admitting an AAT.
Clearly $g_6$ is a transcendence basis of $\mathbb{C}(\Omega \times \Omega )$, so it is algebraic over a field of $\mathcal{P}_6$.
Since the restriction of a translation of each of the latter maps to a sufficiently small neighborhood of $\mathbb{R}^2$ is a diffeomorphism, 
each $(\mathbb{R}^2,+,g_1)$,\ldots , $(\mathbb{R}^2,+,g_6)$ is a locally Nash group.

By Theorem \ref{T1} every simply connected $n$-dimensional abelian locally Nash group is locally Nash isomorphic to one of the form $(\mathbb{R}^2,+,f)$ where $f:\mathbb{C}^2\rightarrow \mathbb{C}^2$ is a real meromorphic map admitting an AAT. Furthermore, by Fact \ref{TP} there exists $i\in \{1,\ldots ,6\}$ and $\mathbb{L}\in \mathcal{P}_i$ such that $f$ is 
algebraic over $\mathbb{L}$. Let $(\mathbb{R}^2,+,g)$ be another locally Nash group and fix $j\in \{1,\ldots ,6\}$ such that there exists $\mathbb{L}'\in \mathcal{P}_j$ 
such that $g$ is algebraic over $\mathbb{L}'$.
It is enough to show if $(\mathbb{R}^2,+,f)$ and $(\mathbb{R}^2,+,g)$ are isomorphic as locally Nash groups then $i=j$.
We recall that $f,g:\mathbb{C}^2\rightarrow \mathbb{C}^2$ are real meromorphic maps and hence, 
by Proposition \ref{different ranks}, $rank\, \Lambda _g =rank\, \Lambda _f$.
Let $r:=rank\, \Lambda _f$.
By Proposition \ref{rank} some of the cases are already solved, namely:
if $r=0$ then $i=j=1$; if $r=1$ then $i=j=2$; if $r=3$ then $i=j=5$; and if $r=4$ then $i=j=6$.
For the case $r=2$, suppose for a contradiction that $i=4$ and $j=3$.
By definition there exist $\alpha_1,\alpha_2\in GL_2(\mathbb{C})$, $a \in \mathbb{C}$ and a lattice $\Omega$ of $(\mathbb{C},+)$ 
such that $f$ is algebraic over $\mathbb{L}_1=\mathbb{C}(g_{4,a ,\Omega}\circ \alpha _1)$ and $g$ is algebraic over 
$\mathbb{L}_2=\mathbb{C}(g_3\circ \alpha _2)$.
By Corollary \ref{compatibilityAut.1} there exists $\alpha \in GL_2(\mathbb{R})$ such that 
$g\circ \alpha$ is algebraic over $\mathbb{R}(f)$.  
Since the coordinate functions of $g \circ \alpha$ are algebraically independent over $\mathbb{C}$, we get that
$g_3\circ \alpha _3$ is algebraic over $\mathbb{C}(g_{4,a,\Omega})$ for some 
$\alpha _3:=\begin{pmatrix} a & b\\ c & d\end{pmatrix}\in GL_2(\mathbb{C})$.
Since $g_{4,a,\Omega}$ is algebraic over $\mathbb{C}(\wp _{\Lambda}(u),\zeta _{\Lambda} (u), v)$, we have that
\[
\left( e^{au+bv}, e^{cu+dv} \right) \text{ is algebraic over } \mathbb{C}(\wp _{\Lambda}(u),\zeta _{\Lambda} (u), v).
\]
Since $\alpha _3\in GL_2(\mathbb{C})$, either $b\neq 0$ or $d\neq 0$.
Without loss of generality we may assume $b\neq 0$.
We also note that $e^{bv}$ is algebraic over $\mathbb{C}(e^{au}, e^{au+bv})$, so
\[
e^{bv} \text{ is algebraic over } \mathbb{C}(e^{au}, \wp _{\Lambda}(u),\zeta _{\Lambda} (u), v),
\]
which means that $e^{bv}$ is algebraic over $\mathbb{C}(v)$, a contradiction.
So either $i=j=3$ or either $i=j=4$.
\end{proof}
 
\section*{Appendix: One-dimensional simply connected locally Nash groups.}\label{One-dimensional}
\setcounter{section}{6}
\setcounter{lemma}{0}

A classification of the one-dimensional simply connected locally Nash groups was given by Madden and Stanton in \cite{Madden_Stanton} 
(see also \cite{Madden_Stanton_Errata}).
In this appendix we provide a detailed proof of such classification using the techniques we have developed for dimension $2$.

Meromorphic functions from $\mathbb{C}$ to $\mathbb{C}$ that admit an algebraic addition theorem were classified by
Weierstrass, see for example \cite[Ch.VII]{Hancock}.

\begin{fact}[Weierstrass]\label{TW}
If $f:\mathbb{C}\rightarrow \mathbb{C}$ is a meromorphic function that admits an AAT then there exists 
$\alpha \in GL_1(\mathbb{C})$ such that $f$ is algebraic over $\mathbb{C}(g\circ \alpha )$, where $g$ is either
$$\hspace{-1cm}\begin{array}{lcll}
\emph{(I)} & \hspace{2cm}& g(u)  = & u, \vspace{1mm}\\
 \emph{(II)} & \hspace{2cm}& g(u) = &e^u, \vspace{1mm} \\
 \emph{(III)} & \hspace{2cm}& g(u)= &\wp _{\Lambda}(u), \text{ for some lattice } \Lambda < (\mathbb{C},+).
\end{array}$$\end{fact}

See e.g. \cite[Ch.II]{Chandrasekharan} for basic properties of $\wp _{\Lambda}$.
Note also that all functions of Fact \ref{TW} admit an AAT.

In Fact \ref{TW}, we obtain that $f\circ \alpha ^{-1}$ is algebraic over $\mathbb{C}(g)$, so under a suitable change of complex coordinates the meromorphic function $f$ is algebraic 
either over $\mathbb{C}(id)$, $\mathbb{C}(exp)$ or $\mathbb{C}(\wp _{\Lambda})$.

\medskip

We begin this section with some technical lemmas.
Firstly, some properties of the Weierstrass $\wp$-function.

\begin{lemma} Let $\Lambda$ be a lattice of $(\mathbb{C},+)$.
Then, $\wp _{\overline{\Lambda}} (u)=\overline{\wp _{\Lambda}(\overline{u})}$.
Hence, $\wp _\Lambda$ is a real meromorphic function if and only if $\Lambda =\overline{\Lambda}$.
\end{lemma}
\begin{proof}
We note that
\[
\overline{\wp _{\Lambda}(u)}
=
\overline{\frac{1}{u^2}+
\sum _{\omega \in \Lambda \setminus \{0\}}
\left\{ \frac{1}{(u-\omega ) ^2}-\frac{1}{\omega ^2} \right\} }
=
\frac{1}{\overline{u}^2}+
\sum _{\omega \in \Lambda \setminus \{0\}}
\left\{ \frac{1}{(\overline{u}-\overline{\omega }) ^2}-\frac{1}{\overline{\omega} ^2}\right\}.
\]
Therefore,
\[
\overline{\wp _{\Lambda}(u)}=
\frac{1}{\overline{u}^2}+
\sum _{\omega \in \overline{\Lambda }\setminus \{0\}}
\left\{ \frac{1}{(\overline{u}-\omega ) ^2}-\frac{1}{\omega ^2}\right\}.
\]
For the second statement recall that since $\wp _\Lambda$ is a meromorphic function it is real if and only if 
$\overline{\wp _{\Lambda}(\overline{u})}=\wp _\Lambda (u)$.
\end{proof}

\begin{lemma}\label{cosets} Let $\Lambda _1$ and $\Lambda _2$ be lattices of $\mathbb{C}$ such that 
$\Lambda _1<\Lambda _2$ and $[\Lambda _2:\Lambda _1]=n$ for some $n\in \mathbb{N}$.
Then, there exist $a_1,\ldots ,a_n,C\in \mathbb{C}$ such that 
$\wp _{\Lambda _2}(u)=\sum _{i=1}^n \wp _{\Lambda _1}(u+a_i)+C$.
\end{lemma}
\begin{proof}
The lemma can be proved by direct computation, which also shows that $C=0$.
However we will prove the lemma in a different way.
It is enough to show that there exist $a_1,\ldots ,a_n\in \mathbb{C}$ such that 
$\wp '_{\Lambda _2}(u)=\sum _{i=1}^n \wp '_{\Lambda _1}(u+a_i)$. Recall that
\[
\wp '_{\Lambda _2}(u)=-2\sum _{\omega \in \Lambda _2}\frac{1}{(u-\omega )^3}. 
\]
Since $[\Lambda _2:\Lambda _1]=n$ there exist $a_1,a_2,\ldots ,a_n\in \Lambda _2$ 
such that $\Lambda _2=\bigcup _{i=1}^n (\Lambda _1+a_i)$.
So
\[
\wp '_{\Lambda _2}(u)=-2\sum _{i=1}^n\sum _{\omega \in \Lambda _1}\frac{1}{(u-(\omega +a_i))^3} 
\]
and hence $\wp '_{\Lambda _2}(u)=\sum _{i=1}^n \wp '_{\Lambda _1}(u-a_i)$.
\end{proof}

\begin{lemma}\label{sublattices} Let $\Lambda _1$ and $\Lambda _2$ be lattices of $(\mathbb{C},+)$ such that 
$\Lambda _1<\Lambda _2$.
Then $\wp _{\Lambda _1}$ and $\wp _{\Lambda _2}$ are algebraically dependent over $\mathbb{C}$.
\end{lemma}
\begin{proof}
Since both $\Lambda _1$ and $\Lambda _2$ are lattices of $(\mathbb{C},+)$, $rank\, \Lambda _1=rank\, \Lambda _2$, 
hence $[\Lambda _2:\Lambda _1]<\infty$.
Let $n=[\Lambda _2:\Lambda _1]$.
By Lemma \ref{cosets} there exist $a_1,\ldots ,a_n,C\in \mathbb{C}$ such that 
$\wp _{\Lambda _2}(u)=\sum _{i=1}^n \wp _{\Lambda _1}(u+a_i)+C$.
Since $\wp _{\Lambda _1}$ admits an AAT and by Corollary \ref{algebraic for functions}.$(1)$, $\wp _{\Lambda _1}(u+a)$ is 
algebraic over $\mathbb{C}(\wp _{\Lambda _1}(u))$ for all $a\in \mathbb{C}$.
So $\wp _{\Lambda _2}$ is algebraic over $\mathbb{C}(\wp _{\Lambda _1})$.
\end{proof}

\begin{lemma}\label{real wp}  Let $\Lambda$ be a lattice of $(\mathbb{C},+)$.
Let $g:\mathbb{C}\rightarrow \mathbb{C}$ be a real meromorphic function such that $\Lambda _g$ is a discrete subgroup 
of $(\mathbb{C},+)$ and $g$ is algebraic over $\mathbb{C}(\wp _\Lambda )$.
Then there exists a real lattice $\Lambda '<\Lambda$ such that $g$ is algebraic over 
$\mathbb{C}(\wp _{\Lambda '})$.
\end{lemma}
\begin{proof}
Since $\Lambda$ is a lattice, $\Lambda _g$ is a real lattice by Lemmas \ref{discrete groups}.$(2)$ 
and \ref{algebraicity of periods}.
Hence, by Corollary \ref{algebraicity of periods2} there exists a real lattice $\Lambda '$ of $(\mathbb{C},+)$
such that $\Lambda '<\Lambda$ and $\Lambda '<\Lambda _g$.
On the other hand $g$ is algebraic over $\mathbb{C}(\wp _\Lambda)$ and $\wp _\Lambda$ is algebraic over 
$\mathbb{C}(\wp _{\Lambda '})$, by Lemma \ref{sublattices}, so $g$ is algebraic over $\mathbb{C}(\wp _{\Lambda '})$.
\end{proof}

Some of the possible locally Nash group structures for $(\mathbb{R},+)$ will be given by Weierstrass $\wp$-functions over lattices
of the form $<1,ai>_{\mathbb{Z}}$ where $a\in \mathbb{R}^*$.
We will use the notation $(\mathbb{R},+,\wp _{\Lambda})$ of the introduction.
\begin{remark}\label{Desrank2}\emph{Note that a nontrivial real discrete subgroup $\Lambda$ of $(\mathbb{C},+)$ is of rank $1$ if it is either of the form 
$<a>_{\mathbb{Z}}$ or $<ia>_{\mathbb{Z}}$ for some $a\in \mathbb{R}$; and it is of rank $2$ if it is has a finite index 
subgroup of the form $<a,bi>_{\mathbb{Z}}$ for some $a,b\in \mathbb{R}^*$.
Indeed, since $\Lambda$ is real we must have $\overline{\lambda}\in \Lambda$ for any $\lambda \in \Lambda$.
The only special case is when $\Lambda=<\lambda ,\overline{\lambda}>_{\mathbb{Z}}$ with $\lambda =a+ib$ with both $a,b\neq 0$.
Then $<2a,2ib>_{\mathbb{Z}}$ is the finite index subgroup of $\Lambda$.}
\end{remark}

\begin{fact}[{\cite[Theorem 2]{Madden_Stanton}}]\label{rational wp} 
Let $a,b\in \mathbb{R}^*$ and let $\Lambda _1:=<1,ia>_{\mathbb{Z}}$ and $\Lambda _2:=<1,ib>_{\mathbb{Z}}$.
Then $(\mathbb{R},+,\wp _{\Lambda _1})$ and $(\mathbb{R},+,\wp _{\Lambda _2})$
are isomorphic as locally Nash groups if and only if $ab^{-1}\in \mathbb{Q}$.
\end{fact}
\begin{proof}
Firstly, suppose that there are $m,n\in \mathbb{Z}^*$ such that $mn^{-1}=ab^{-1}$.
Let $\Lambda :=<1,ina>_{\mathbb{Z}}$.
Hence, both $\Lambda <\Lambda _1$ and $\Lambda <\Lambda _2$.
So, by Lemma \ref{sublattices}, $\wp _{\Lambda}$ is algebraic over both $\mathbb{C}(\wp _{\Lambda _1})$ and 
$\mathbb{C}(\wp _{\Lambda _2})$.
Hence $\wp _{\Lambda _1}$ is algebraic over $\mathbb{C}(\wp _{\Lambda _2})$.
Since $\wp _{\Lambda _1}$ and $\wp _{\Lambda _2}$ are real meromorphic functions, 
$\wp _{\Lambda _1}$ is algebraic over $\mathbb{R}(\wp _{\Lambda _2})$. 
So $(\mathbb{R},+,\wp _{\Lambda _1})$ and $(\mathbb{R},+,\wp _{\Lambda _2})$ are locally Nash isomorphism by 
Corollary \ref{compatibilityAut.1}.

Now, suppose that $(\mathbb{R},+,\wp _{\Lambda _1})$ and $(\mathbb{R},+,\wp _{\Lambda _2})$ are isomorphic as locally Nash groups.
By Corollary \ref{compatibilityAut.1}, there exists $\alpha \in GL_1(\mathbb{R})$ such that 
\[
\wp _{\Lambda _2} \circ \alpha \text{ is algebraic over } \mathbb{R}(\wp _{\Lambda _1}).
\]
Let $c$ denote the unique element of $\mathbb{R}^*$ such that
\[
\alpha :\mathbb{R}\rightarrow \mathbb{R}: x\mapsto cx.
\]
Let $\Lambda _2 ':=\alpha ^{-1}(\Lambda _2)$.
Then $\Lambda _2'=<c^{-1},ibc^{-1}>_{\mathbb{Z}}$.
We note that $\Lambda _2'$ is the group of periods of $\wp _{\Lambda _2} \circ \alpha$.
By Corollary \ref{algebraicity of periods2} there exists a real lattice $\Lambda$ of $(\mathbb{C},+)$ such that 
both $\Lambda <\Lambda _1$ and $\Lambda <\Lambda _2'$.
By Remark \ref{Desrank2} we may assume that there exist $n_1,n_2,m_1,m_2\in \mathbb{N}$ such that 
\[
\Lambda =<n_1,n_2ia>_{\mathbb{Z}}=<m_1c^{-1},m_2ibc^{-1}>_{\mathbb{Z}}.
\]
So $m_1c^{-1}=\ell n_1$ for some $\ell \in \mathbb{Z}$ and hence $c\in \mathbb{Q}$.
Also $n_2ia=\ell m_2ibc^{-1}$ for some $\ell \in \mathbb{Z}$ and hence $ab^{-1}\in \mathbb{Q}$.
\end{proof}

Now we prove \cite[Theorem 1]{Madden_Stanton} from a different point of view that involves ranks of lattices. 
We will use the notation $(\mathbb{R},+,f )$ introduced before Theorem \ref{T1} (in particular we recall that the map associated to a chart 
of the identity can be a translate of $f$).

\begin{theorem}[{\cite[Theorem 1]{Madden_Stanton}}]\label{1dim classification}
Every simply connected one-dimensional locally Nash group is isomorphic as a locally Nash group to one of the following locally 
Nash groups.
\begin{enumerate}
 \item[$(1)$] $(\mathbb{R},+,id)$.
 \item[$(2)$] $(\mathbb{R},+,exp)$.
 \item[$(3)$] $(\mathbb{R},+,sin)$.
 \item[$(4)$] $(\mathbb{R},+,\wp _\Lambda )$ where $\Lambda =<1,ia>_{\mathbb{Z}}$ for some $a\in \mathbb{R}^*$.
\end{enumerate}
The three first ones are not locally Nash group isomorphic to each other and neither they are isomorphic to one of the fourth type.
$(\mathbb{R},+,\wp _{<1,ia>_{\mathbb{Z}}})$ and $(\mathbb{R},+,\wp _{<1,ib>_{\mathbb{Z}}})$
are isomorphic as locally Nash groups if and only if $a/b\in \mathbb{Q}$.
\end{theorem}
\begin{proof}We first note that by Lemma \ref{AAT-star} each of the four cases are indeed 
locally Nash groups. Every connected analytic group of dimension $1$ is abelian, so by Theorem \ref{T1} every simply connected one-dimensional locally Nash group is isomorphic as a locally Nash group to 
some $(\mathbb{R},+,f)$, where $f:\mathbb{C}\rightarrow \mathbb{C}$ is a real meromorphic function admitting an AAT. We are in the hypothesis of Fact \ref{TW} and therefore there exists 
$\alpha \in GL_1(\mathbb{C})$ such that $f$ is algebraic over $\mathbb{C}(g\circ \alpha )$, where 
$g:\mathbb{C}\rightarrow \mathbb{C}$ is either $id$ or $exp$ or $\wp _{\Lambda}$ for some lattice $\Lambda$ of $(\mathbb{C},+)$.
Let $c\in \mathbb{C}^*$ be such that $\alpha :\mathbb{C}\rightarrow \mathbb{C}: u\mapsto cu$.
We note that $\Lambda _f$ is a real discrete subgroup of $(\mathbb{C},+)$ by 
Lemma \ref{discrete groups}.$(3)$, also that, by Lemma \ref{discrete groups}.$(4)$, 
\[
rank\, \Lambda _{id\circ \alpha}=0, \quad rank\, \Lambda_{exp\circ \alpha}=1,\quad rank\, \Lambda_{\wp _\Lambda \circ \alpha}=2.
\]

\smallskip

{\it Case I: $rank\, \Lambda _f=0$.}
Then, by Lemma \ref{algebraicity of periods}, $rank\, \Lambda _{g\circ \alpha}=0$.
So $f$ is algebraic over $\mathbb{C}(c \cdot id)=\mathbb{C}(id)$ and therefore by Corollary \ref{compatibilityAut.1}, $(\mathbb{R},+,f)$ and $(\mathbb{R},+,id)$ are isomorphic as 
locally Nash groups, what gives us $(1)$ in the statement of the theorem.

\smallskip

{\it Case II: $rank\, \Lambda _f=1$.}
Then, by Lemma \ref{algebraicity of periods}, $rank\, \Lambda _{g\circ \alpha}=1$.
So, $g=exp$ and hence $f$ is algebraic over $\mathbb{C}(exp \circ \alpha)$.
By Remark \ref{Desrank2} there exists $a\in \mathbb{R}^*$ such that 
either $\Lambda _g=<a>_{\mathbb{Z}}$ or either $\Lambda _g=<ia>_{\mathbb{Z}}$.\\
{\it Subcase II.1: $\Lambda _g=<ia>_{\mathbb{Z}}$.}
In this case, $f(u)$ is algebraic over $\mathbb{C}(e^{2\pi u/a})$.
Since both $f(u)$ and $u\mapsto e^{2\pi u/a}$ are real meromorphic functions, we get by Corollary \ref{compatibilityAut.1} that
$(\mathbb{R},+,f)$ and $(\mathbb{R},+,x\mapsto e^{2\pi x/a})$ are isomorphic as locally Nash groups.
Let $\tilde \alpha (x):=ax/2\pi\in GL_1(\mathbb{R})$.
Then again, by Corollary \ref{compatibilityAut.1} applied to $\tilde \alpha$ we deduce that $(\mathbb{R},+,exp)$ and $(\mathbb{R},+,x\mapsto e^{2\pi x/a})$ 
are isomorphic as locally Nash groups.
So $(\mathbb{R},+,f)$ is locally Nash isomorphic to $(\mathbb{R},+,exp)$, what gives us $(2)$ in the statement of the 
theorem.\\
{\it Subcase II.2: $\Lambda _g=<a>_{\mathbb{Z}}$.}
In this case $f(u)$ is algebraic over $\mathbb{C}(e^{2\pi iu/a})$.
Hence, $f(u)$ is algebraic over $\mathbb{R}(sin (2\pi u/a))$.
Hence applying Corollary \ref{compatibilityAut.1} we deduce that $(\mathbb{R},+,f)$ and $(\mathbb{R},+,x\mapsto sin (2\pi x/a))$ 
are isomorphic as locally Nash groups. Again by Corollary \ref{compatibilityAut.1} applied to $\tilde\alpha$ above we get 
 that $(\mathbb{R},+,sin)$ and $(\mathbb{R},+,x\mapsto sin(2\pi x/a))$ are isomorphic as locally Nash groups.
So $(\mathbb{R},+,f)$ is locally Nash isomorphic to $(\mathbb{R},+,sin)$, what gives us $(3)$ in the statement of the theorem.

\smallskip
 
{\it Case 3: $rank\, \Lambda _f=2$.}
Then, by Lemma \ref{algebraicity of periods}, $rank\, \Lambda _{g\circ \alpha}=2$.
So there exists a lattice $\Lambda$ of $(\mathbb{C},+)$ such that  $g=\wp _{\Lambda}$ and hence $f(u)$ is algebraic over 
$\mathbb{C}(\wp _\Lambda (cu))$.
Since $\wp _\Lambda (cu)=c^{-2}\wp _{c^{-1}\Lambda}(u)$ and by Lemma \ref{real wp}, $f$ is algebraic over $\mathbb{C}(\wp _\Lambda)$ for some real lattice $\Lambda$ of $(\mathbb{C},+)$. Moreover, by Lemma \ref{sublattices} and Remark \ref{Desrank2}, we may assume 
that $\Lambda$ is of the form $<a,ib>_{\mathbb{Z}}$ for some $a,b\in \mathbb{R}^*$.
Hence applying Corollary \ref{compatibilityAut.1} we deduce that $(\mathbb{R},+,f)$ and $(\mathbb{R},+,\wp _\Lambda)$ are 
isomorphic as locally Nash groups.
Let $\Lambda':=<1,ib/a>_{\mathbb{Z}}$ and let $\tilde \alpha (x):=a^{-1}x\in GL_1(\mathbb{R})$.
We note that $\wp _{\Lambda '}(\tilde \alpha (x))=a^2\wp _{\Lambda }(x)$ and therefore by Corollary \ref{compatibilityAut.1} applied to $\tilde \alpha $ we deduce that $(\mathbb{R},+,\wp _{\Lambda})$
and $(\mathbb{R},+, \wp _{\Lambda '})$ are isomorphic as locally Nash groups.
So in this case $(\mathbb{R},+,f)$ is locally Nash isomorphic to $(\mathbb{R},+,\wp _{\Lambda '})$ where 
$\Lambda ' =<1,ia>_{\mathbb{Z}}$ for some $a\in \mathbb{R}^*$, what gives us $(4)$ in the statement of the theorem.

\smallskip

Now we show that the four types of groups considered are not isomorphic as locally Nash groups.
By Proposition \ref{different ranks} the only ones that can be isomorphic as locally Nash groups are of the type $(2)$ and $(3)$ or
both of the type $(4)$.
Suppose $(\mathbb{R},+,exp)$ and $(\mathbb{R},+,sin)$ are isomorphic.
Then by Corollary \ref{compatibilityAut.1} there exists $\alpha \in GL_1(\mathbb{R})$ such that 
$x\mapsto e^{\alpha (x)}$ is algebraic over $\mathbb{C}(x\mapsto sin(x))$.
Since the periods of $x\mapsto e^x$ are imaginary, the periods of $x\mapsto sin (x)$ are real numbers and $\alpha$ cannot 
map imaginary numbers into real numbers, this contradicts Lemma \ref{algebraicity of periods}.

\smallskip

The last statement about  groups of the fourth type follows from Fact \ref{rational wp}.
\end{proof}

\bibliographystyle{plain}
\bibliography{research}

\begin{thebibliography}{10}

\bibitem{Adams}
J.~Frank Adams.
\newblock {\em Lectures on {L}ie groups}.
\newblock W. A. Benjamin, Inc., New York-Amsterdam, 1969.

\bibitem{Bochnak_Coste_Roy}
Jacek Bochnak, Michel Coste, and Marie-Fran{\c{c}}oise Roy.
\newblock {\em Real algebraic geometry}, volume~36 of {\em Ergebnisse der
  Mathematik und ihrer Grenzgebiete (3) [Results in Mathematics and Related
  Areas (3)]}.
\newblock Springer-Verlag, Berlin, 1998.
\newblock Translated from the 1987 French original, Revised by the authors.

\bibitem{Bochner_Martin}
Salomon Bochner and William~Ted Martin.
\newblock {\em Several {C}omplex {V}ariables}.
\newblock Princeton Mathematical Series, vol. 10. Princeton University Press,
  Princeton, N. J., 1948.

\bibitem{Chandrasekharan}
K.~Chandrasekharan.
\newblock {\em Elliptic functions}, volume 281 of {\em Grundlehren der
  Mathematischen Wissenschaften [Fundamental Principles of Mathematical
  Sciences]}.
\newblock Springer-Verlag, Berlin, 1985.

\bibitem{Fernando_Gamboa_Ruiz}
Jos{\'e}~F. Fernando, J.~M. Gamboa, and Jes{\'u}s~M. Ruiz.
\newblock Finiteness problems on {N}ash manifolds and {N}ash sets.
\newblock {\em J. Eur. Math. Soc. (JEMS)}, 16(3):537--570, 2014.

\bibitem{Gunning_Rossi}
Robert~C. Gunning and Hugo Rossi.
\newblock {\em Analytic functions of several complex variables}.
\newblock Prentice-Hall, Inc., Englewood Cliffs, N.J., 1965.

\bibitem{Hancock}
Harris Hancock.
\newblock {\em {Lectures on the theory of elliptic functions}}.
\newblock Wiley, New York, NY, 1910.

\bibitem{Hrushovski_Pillay}
Ehud Hrushovski and Anand Pillay.
\newblock Groups definable in local fields and pseudo-finite fields.
\newblock {\em Israel J. Math.}, 85(1-3):203--262, 1994.

\bibitem{Hrushovski_Pillay_Errata}
Ehud Hrushovski and Anand Pillay.
\newblock Affine {N}ash groups over real closed fields.
\newblock {\em Confluentes Math.}, 3(4):577--585, 2011.

\bibitem{Madden_Stanton}
James~J. Madden.
\newblock Correction to: ``{O}ne-dimensional {N}ash groups'' [{P}acific {J}.\
  {M}ath.\ {\bf 154} (1992), no.\ 2, 331--344; {MR}1159515 (93d:14087)] by
  {M}adden and {C}. {M}. {S}tanton.
\newblock {\em Pacific J. Math.}, 161(2):393, 1993.

\bibitem{Madden_Stanton_Errata}
James~J. Madden and Charles~M. Stanton.
\newblock One-dimensional {N}ash groups.
\newblock {\em Pacific J. Math.}, 154(2):331--344, 1992.

\bibitem{Narashiman}
Raghavan Narasimhan.
\newblock {\em Introduction to the theory of analytic spaces}.
\newblock Lecture Notes in Mathematics, No. 25. Springer-Verlag, Berlin-New
  York, 1966.

\bibitem{Nash}
John Nash.
\newblock Real algebraic manifolds.
\newblock {\em Ann. of Math. (2)}, 56:405--421, 1952.

\bibitem{Newns}
W.~F. Newns.
\newblock Functional dependence.
\newblock {\em Amer. Math. Monthly}, 74:911--920, 1967.

\bibitem{Otero}
Margarita Otero.
\newblock A survey on groups definable in o-minimal structures.
\newblock In {\em Model theory with applications to algebra and analysis.
  {V}ol. 2}, volume 350 of {\em London Math. Soc. Lecture Note Ser.}, pages
  177--206. Cambridge Univ. Press, Cambridge, 2008.

\bibitem{Painleve}
Paul Painlev{\'e}.
\newblock Sur les fonctions qui admettent un th\'eor\`eme d'addition.
\newblock {\em Acta Math.}, 27(1):1--54, 1903.

\bibitem{Shiota1}
Masahiro Shiota.
\newblock {\em Nash manifolds}, volume 1269 of {\em Lecture Notes in
  Mathematics}.
\newblock Springer-Verlag, Berlin, 1987.

\bibitem{Shiota2}
Masahiro Shiota.
\newblock Nash functions and manifolds.
\newblock In {\em Lectures in real geometry ({M}adrid, 1994)}, volume~23 of
  {\em de Gruyter Exp. Math.}, pages 69--112. de Gruyter, Berlin, 1996.

\bibitem{Siegel}
C.~L. Siegel.
\newblock {\em Topics in complex function theory. {V}ol. {III}}.
\newblock Wiley Classics Library. John Wiley \& Sons, Inc., New York, 1989.
\newblock Abelian functions and modular functions of several variables,
  Translated from the German by E. Gottschling and M. Tretkoff, With a preface
  by Wilhelm Magnus, Reprint of the 1973 original, A Wiley-Interscience
  Publication.

\bibitem{van_den_Dries}
Lou van~den Dries.
\newblock {\em Tame topology and o-minimal structures}, volume 248 of {\em
  London Mathematical Society Lecture Note Series}.
\newblock Cambridge University Press, Cambridge, 1998.

\end{thebibliography}

\end{document}